\documentclass[a4paper,12pt]{amsart}
\usepackage{amscd,amsfonts,amssymb}
\usepackage{tikz}
\usepackage{xcolor}
\usepackage{pgfplots}
\usetikzlibrary{automata,positioning,calc,trees}
\usetikzlibrary{intersections,pgfplots.fillbetween}
\usepackage{graphicx,tikz}
\usepackage{pgf,tikz,pgfplots}
\usepackage{mathrsfs}
\usetikzlibrary{arrows}
\usepackage{enumerate}
\usepackage[shortlabels]{enumitem}
\usepackage{mathrsfs}
\usepackage{amssymb,amsmath,amsthm,color}
\usepackage{caption,subcaption}
\usepackage{hyperref}
\usepackage{url}
\usepackage{setspace}
\usepackage{float}
\date{\today}

\allowdisplaybreaks

\newtheorem{theorem}{Theorem}[section]
\newtheorem{proposition}[theorem]{Proposition}
\newtheorem{lemma}[theorem]{Lemma}

\newtheorem{remark}[theorem]{Remark}

\def\be#1 {\begin{equation} \label{#1}}
\newcommand{\ee}{\end{equation}}

\def\sqw{\hbox{\rlap{\leavevmode\raise.3ex\hbox{$\sqcap$}}$%
\sqcup$}}
\def\findem{\ifmmode\sqw\else{\ifhmode\unskip\fi\nobreak\hfil
\penalty50\hskip1em\null\nobreak\hfil\sqw
\parfillskip=0pt\finalhyphendemerits=0\endgraf}\fi}

\newcommand{\R}{{\mathbb {R}}}

\newcommand{\N}{{\mathbb N}}
\newcommand{\Z}{{\mathbb Z}}

  \newcommand{\blue}{\textcolor{blue}}

\author{Jotsaroop Kaur and Saurabh Shrivastava}
\address{
}
\email{}
\address[Jotsaroop Kaur]{
Department of Mathematics\\
Indian Institute Science Education and Research\\
 Mohali, India}
\email{jotsaroop@iisermohali.ac.in}

\address [Saurabh Shrivastava]{
Department of Mathematics\\
Indian Institute Science Education and Research Bhopal\\
Bhopal-462066, India}
\email{saurabhk@iiserb.ac.in}

\begin{document}

\title[Bilinear Bochner-Riesz maximal function]{Maximal estimates for  Bilinear Bochner-Riesz means}

\begin{abstract} We establish improved and sharp $L^p$ estimates for the maximal bilinear Bochner-Riesz means in all dimensions $n\geq 1$. This work extends the results proved by Jeong and Lee~\cite{JL}. We also recover the known results for the bilinear Bochner-Riesz means. The method of proof involves a new decomposition of the bilinear Bochner-Riesz multiplier $(1-|\xi|^2-|\eta|^2)_+^{\alpha}$ and delicate analysis in proving $L^p$ estimates for frequency localized square functions. 
\end{abstract}
\maketitle
\tableofcontents
\section{Introduction}
%%%%%%%%%%%%%%%%
%Let $f$ and $g$ be integrable functions defined on $[0,1]^n, n\geq 1.$ 
The convergence of the partial sum operators corresponding to two fold product of $n-$ dimensional Fourier series
\begin{eqnarray*}\label{con1}
\sum_{|m|^2+|k|^2\leq R^2} \hat f(m) \hat{g}(k)e^{2\pi i(k+m)\cdot x}
\end{eqnarray*} 
to $f(x)g(x)$ as $R\rightarrow \infty$ is a classical problem in Fourier analysis. Here $f$ and $g$ are $1-$periodic functions with respect to each variable belonging to Lebesgue spaces $L^{p_1}([0,1]^n)$ and $L^{p_2}([0,1]^n)$ respectively, where $ 1\leq p_1, p_2\leq \infty$.

This motivates us to consider the following more general partial sum operators 
\begin{eqnarray}\label{con}
\sum_{|m|^2+|k|^2\leq R^2} \left(1-\frac{|m|^2+|k|^2}{R^2}\right)^{\alpha}\hat f(m) \hat{g}(k)e^{2\pi i(k+m)\cdot x},~~\text{where}~~\alpha\geq 0.
\end{eqnarray}
The partial sum operators are referred to as the bilinear Bochner-Riesz means of order $\alpha$. 

The transference principle for bilinear multipliers establishes that the study of $L^p$ boundedness of the bilinear Bochner-Riesz means is equivalent to studying the corresponding estimates for the bilinear Bochner-Riesz operator of index $\alpha\geq 0$ defined on $\R^n$ by
$$\mathcal{B}^{\alpha}_R(f,g)(x)=\int_{\R^{n}}\int_{\R^{n}}\left(1-\frac{|\xi|^2+|\eta|^2}{R^2}\right)^{\alpha}_{+}\hat{f}(\xi)\hat{g}(\eta)e^{2\pi ix.(\xi+\eta)}d\xi d\eta,$$
where $f,g\in \mathcal{S}(\R^n)$ and $\hat f$ denotes the Fourier transform of $f$ given by $\hat{f}(\xi)=\int_{\R^n}f(x)e^{-2\pi ix.\xi}dx$. When $R=1$, we simply use the notation $\mathcal{B}^{\alpha}$. 

In this paper we are interested in investigating the a.e. convergence of $\mathcal{B}^{\alpha}_R(f,g)$ as $R\rightarrow\infty$ for $(f,g)$ in $L^{p_1}(\R^n)\times L^{p_2}(\R^n)$ for $1\leq p_1,p_2\leq \infty$. In view of the general principle due to Stein we know that in order to study the a.e. convergence of $\mathcal{B}^{\alpha}_R(f,g)$ it is enough to prove suitable $L^p$ estimates for the corresponding maximal function
$$\mathcal B^{\alpha}_*(f,g)(x)=:\sup_{R>0}|\mathcal{B}^{\alpha}_R(f,g)(x)|.$$
More precisely, for $f, g\in \mathcal S(\R^n)$ we are concerned with the following estimates 
\begin{eqnarray}\label{mbre}
\|\mathcal B_*^{\alpha}(f,g)\|_{L^p(\R^n)}\lesssim \|f\|_{L^{p_1}(\R^n)} \|g\|_{L^{p_2}(\R^n)}.
\end{eqnarray}
Also, the corresponding estimate for the bilinear Bochner-Riesz mean
\begin{eqnarray}\label{bre}
\|\mathcal B^{\alpha}(f,g)\|_{L^p(\R^n)}\lesssim \|f\|_{L^{p_1}(\R^n)} \|g\|_{L^{p_2}(\R^n)}.
\end{eqnarray}
Here we shall always assume that the exponents $p_1, p_2, p$ satisfy $1\leq p_1, p_2\leq \infty$ and $\frac{1}{p_1}+\frac{1}{p_2}=\frac{1}{p}$. The notation $A\lesssim B$ in the estimates above means that there is an implicit constant $C>0$ such that $A\leq CB.$ The constant $C$ may depend on $\alpha, n,p_1$ and $p_2.$ We will not keep track of such constants appearing in various inequalities unless we need them in order for the proof to work. In some inequalities in the later part of paper we shall also use the notation $A\lesssim_{a,b} B$ to emphasize the dependence of the implicit constant on the parameters $a$ and $b$. 

If the estimate \eqref{mbre} (or \eqref{bre}) holds we say that the operator $\mathcal{B}^{\alpha}_*$ (or $\mathcal{B}^{\alpha}$) is bounded from $L^{p_1}(\R^n)\times L^{p_2}(\R^n)\rightarrow L^{p}(\R^n)$. Observe that the standard dilation arguments tell us that $\mathcal{B}^{\alpha}:L^{p_1}(\R^n)\times L^{p_2}(\R^n)\rightarrow L^{p}(\R^n)$ if, and only if $\mathcal{B}^{\alpha}_R:L^{p_1}(\R^n)\times L^{p_2}(\R^n)\rightarrow L^{p}(\R^n)$ where $R>0$ and  $\frac{1}{p_1}+\frac{1}{p_2}=\frac{1}{p}$. 

The question of $L^{p_1}(\R^n)\times L^{p_2}(\R^n)\rightarrow L^{p}(\R^n)$ boundedness for the operators $\mathcal{B}^{\alpha}$ and $\mathcal{B}^{\alpha}_*$ is referred to as the bilinear Bochner-Riesz probem in general. This is one of the outstanding open problems in the theory of bilinear Fourier multipliers. In this paper we make significant progress on the bilinear Bochner-Riesz problem for $\mathcal{B}^{\alpha}_*$. We establish new and sharp results for the maximal function $\mathcal{B}^{\alpha}_*$.

Note that the operator $\mathcal{B}^{\alpha}_R$ takes the following form in the spatial coordinates. 
$$\mathcal{B}^{\alpha}_R(f,g)(x) := \int_{\mathbb{R}^{n}}\int_{\mathbb{R}^{n}}K^{\alpha}_R(y,z)f(x-y)g(x-z) dydz,$$
where 
\begin{eqnarray*}K^{\alpha}_R(y,z)&=&\int_{\mathbb{R}^{n}}\int_{\mathbb{R}^{n}}\left(1-\frac{\vert \xi\vert^{2}+\vert \eta\vert^{2}}{R^2}\right)^{\alpha}_{+}e^{2\pi \iota (y\cdot\xi+z\cdot\eta)}d\xi d\eta\\
&=&c_{n+\alpha} R^{2n}\frac{ J_{\alpha+n} (2\pi |(Ry,Rz)|)} {|(Ry,Rz)|^{\alpha+n}},~y,z\in \R^n.
\end{eqnarray*}
Here $J_{\alpha+n}$ denotes the standard Bessel function of order $\alpha+n$. 

Using the asymptotic estimates of Bessel functions it is easily verified that for $\alpha>n-\frac{1}{2}$, the kernel $K^{\alpha}_R$ is an integrable function with a uniform constant with respect to $R$. Consequently, the estimates  (\ref{mbre}) and (\ref{bre}) follow for all $1<p_1, p_2\leq \infty$ and $\alpha>n-\frac{1}{2}$, see \cite{BGSY, JSK} for detail. The index $\alpha=n-\frac{1}{2}$ is referred to as the critical index for the bilinear Bochner-Riesz problem.  Therefore, the bilinear Bochner-Riesz problem concerning estimates~(\ref{mbre}) and (\ref{bre}) needs to be investigated when $0\leq \alpha\leq n-\frac{1}{2}$. 

When $\alpha=0$, the operator $\mathcal B^{\alpha}$ is denoted by $\mathcal B$ and it is referred to as the bilinear ball multiplier operator. In~\cite{GL} Grafakos and Li proved the estimate (\ref{bre}) in dimension $n=1$ when $\alpha=0$ and $2\leq p_1,p_2,p'<\infty$ satisfying the H\"{o}lder relation $\frac{1}{p_1}+\frac{1}{p_2}=\frac{1}{p}.$ Here $p'$ denotes the conjugate index to $p$ given by $\frac{1}{p}+\frac{1}{p'}=1$. As mentioned previously in this paper we shall always assume that the exponents $p_1,p_2$ and $p$ satisfy the H\"{o}lder relation unless specified otherwise. We will denote such a triplet by $(p_1,p_2,p)$. In \cite{GL} authors also proved that $\mathcal B$ fails to satisfy (\ref{bre}) for triplets $(2,2,1), (2,\infty,2)$ and $(\infty,2,2)$ when $\alpha=0$. We do not know of any positive result for the bilinear ball multiplier in dimension $n=1$ for exponents lying outside the range mentioned as above. Later,  Diestel and Grafakos~\cite{DF} proved a negative result for boundedness of the operator $\mathcal B$ when $n\geq 2$. They exploited Fefferman's argument for the ball multiplier problem~\cite{Fef} to show that in dimension $n\geq 2$ the estimate (\ref{bre}) cannot hold true for $\alpha=0$ if exactly one of $p_1, p_2,p'$ is less than $2.$ 

The study of $L^p$ boundedness properties of the operator $\mathcal B^{\alpha}$ for $\alpha>0$ was initiated in~\cite{BGSY} and they proved several positive and negative results for the boundedness of the bilinear Bochner-Riesz means. In particular, they obtained almost complete result in dimension $n=1$ proving that (\ref{bre}) holds for $\alpha>0$ for all $1\leq p_1, p_2, p\leq \infty$, see Theorem 4.1 in~\cite{BGSY} for detail. This range of exponents is referred to as the Banach triangle in the theory of bilinear multipliers. Observe that the exponent $p$ has its natural range as $1/2\leq p\leq \infty$ when we allow $p_1$ and $p_2$ in the range $1\leq p_1, p_2\leq \infty.$ Later, in~\cite{LW} authors extended these results, specifically in the non-Banach traingle (i.e. when $p<1$) to an improved range of $p_1,p_2, p$ and $\alpha$. 

Recently, Jeong, Lee, and Vargas~\cite{JLV} further improved the range of exponents when $p_1, p_2\geq 2$ for the estimate (\ref{bre}). They introduced a new approach to address the bilinear Bochner-Riesz problem in dimension $n\geq 2$. Their approach relies on a new decompostion of the bilinear Bochner-Riesz operator into a product of square functions in the pointwise sense. This allows them to obtain new results for the estimate (\ref{bre}) when $n\geq 2$ and $p_1,p_2\geq 2$, see [Section 3, \cite{JLV}] for detail. In particular, in~\cite{JLV} authors obtained new results by improving the lower bounds on the index $\alpha$. They  proved optimal result for the estimate (\ref{bre}) when $p_1=p_2=2$ and $\alpha>0$ for all $n\geq 2$, see Corollary 1.3 in ~\cite{JLV} (also see Proposition~4.10 in~\cite{BGSY}). However, the results in other cases are not optimal. We would like to remark here that their method requires the condition that $p_1,p_2\geq 2$. We will see that with our approach we can recover these results. Moreover, our method allows us to obtain $L^p$ estimates for $\mathcal B^{\alpha}$ when the exponents are less than $2$. In particular, the estimate~\eqref{bre} for exponents in non-Banach triangle are new in dimension $n=1$, see Theorem~\ref{dim1}. However, for $n\geq 2$ we do not get improved results for the estimate~\eqref{bre}  when $p_1, p_2<2$ as compared to already known results in~\cite{BGSY, LW}. 

Now, coming back to the main theme of the paper let us discuss the known results for the bilinear Bochner-Riesz maximal function concerning the estimate~(\ref{mbre}). The study of $L^p$ boundedness properties of $\mathcal B^{\alpha}_*$ has been initiated in~\cite{GHH, He}. In~\cite{GHH} authors proved boundedness of maximal bilinear multiplier operator with a certain decay condition on the multiplier symbol. As an application of their result they proved estimate (\ref{mbre}) for $p_1=p_2=2$ and $\alpha>\frac{2n+3}{4}$. Note that this result is interesting only if $\frac{2n+3}{4}$ is smaller or equal to the critical index $\alpha=n-\frac{1}{2}$. Very recently, Jeong and Lee~\cite{JL} have obtained improved results in this direction generalizing the previously known results significantly. In order to prove $L^p$ estimates in~\cite{JL} authors have exploited the decomposition of the bilinear Bochner-Riesz operators into square functions as carried out in~\cite{JLV}. In particular, they proved that the estimate (\ref{mbre}) holds for $\alpha>1$ for $p_1=p_2=2$. Note that the condition $\alpha>1$ is far from being sharp. 

In this paper our main goal is to improve the range of $p_1,p_2, p$ and $\alpha$ in the maximal estimate~\eqref{mbre}. We establish a new decomposition of the bilinear Bochner-Riesz operators in terms of product of square functions which are closely related with the classical Bochner-Riesz square functions (see section~\ref{sec:dec}  for precise detail).
As a result, for each piece of the decomposition parametrised by $j$, the corresponding bilinear operator can be written as superposition of product of localised linear operators $B_{j,\beta}^{R,t}f$ and $B_t^{\delta}g$. For more details please refer to \eqref{decomope}. 
The operator $B_{j,\beta}^{R,t}$ turns out to be a slight perturbation of the localised Bochner Riesz operator. A further careful analysis allows us to reduce the $L^p$ boundedness of the corresponding square function to the known case of the localised square function of Bochner Riesz operator as defined in \cite{Leesqr}. This allows us to improve the result of E. Jeong and S. Lee \cite{JL} for the boundedness of maximal Bilinear Bochner Riesz function.  See \eqref{better} and equation $4.2$ in Proposition  $4.1$ in \cite {JL} for comparision.

 This approach allows us to establish new and sharp results for the estimates (\ref{mbre}) and (\ref{bre}) with an improved range of exponents $p_1, p_2, p$ and lower bound on $\alpha$. Moreover, it is valid uniformly in all dimensions $n\geq 1$. In particular, we obtain optimal results for $p_1=p_2=2$ when $n\geq 1$, see  Theorems~\ref{maintheorem} and \ref{dim1}. We remark that  the maximal estimate ~\eqref{mbre} in dimension $n=1$ is not known. Our approach allows us to address this case and we obtain new and sharp results for the maximal estimate~\eqref{mbre} in dimension $n=1$, see Theorem~\ref{dim1}. 
\subsection*{Organization of the paper} In Section~\ref{sec:main}  we state the main results and discuss the methodology of our proofs. In this section we also  make a comparision of our results with the known results. In Section~\ref{sec:dec} we present the proof of decomposition of bilinear Bochner-Riesz operator. In Section~\ref{results:used} we develop a discussion about Bochner-Riesz square function and some of the results stated in this section will be used in proving the main results of this paper.  Section~\ref{proof:max} is devoted to proving Theorem~\ref{maintheorem} concerning the boundedness of the bilinear Bochner-Riesz maximal function in dimension $n\geq 2$. Finally, in Section~\ref{proof:br} we present the proof of Theorem~\ref{dim1} which addresses the $L^p$ boundedness of the bilinear Bochner-Riesz maximal function in dimension $n=1$. 
%%%%%%%%%%%%%%%%%%%%%%%%%%%
\section{Main results and methodology}\label{sec:main}
We first set some notation that are required in order to describe our results. 
For $1\leq p\leq \infty$ and $n\geq 1$ denote $\alpha(p)= \max\left\{n|\frac{1}{p}-\frac{1}{2}|-\frac{1}{2},0\right\}$. Note that for $n=1$ we have $\alpha(p)=0$ for all $1\leq p\leq \infty$. For $n\geq2,$ define $p_0(n)=2+\frac{12}{4n-6-k}$ where $n\equiv k~\text{mod}~3, k=0,1,2$.
Denote $\mathfrak p_n=\text{min}~\left\{p_0(n),\frac{2(n+2)}{n}\right\}.$ For $n\geq 1$ and $1\leq p_1,p_2\leq \infty$ we define  
\begin{equation*}
	 \alpha_*(p_1,p_2)=\begin{cases}\alpha(p_1)+\alpha(p_2)&\textup{when}~~ \mathfrak p_n\leq p_1,p_2\leq\infty;\\\\
	 \alpha(p_1)+\left(\frac{1-2p_2^{-1}}{1-2(\mathfrak p_n)^{-1}}\right)\alpha(\mathfrak p_n)&	\textup{when}~~ \mathfrak p_n\leq p_1\leq\infty~ \text{and}~2\leq p_2<\mathfrak p_n;\\\\
	\left(\frac{1-2p_1^{-1}}{1-2(\mathfrak p_n)^{-1}}\right) \alpha(\mathfrak p_n)+\alpha(p_2)&	\textup{when}~~ 2\leq p_1<\mathfrak p_n~ \text{and}~\mathfrak p_n\leq p_2\leq\infty;\\\\
	\left(\frac{2-2p_1^{-1}-2p_2^{-1}}{1-2(\mathfrak p_n)^{-1}}\right)\alpha(\mathfrak p_n)&\textup{when}~~2\leq p_1,p_2 <\mathfrak p_n.
	\end{cases}
\end{equation*}
The following theorems concerning the  $L^p$ boundedness of the bilinear Bochner-Riesz maximal function $\mathcal B_*^{\alpha}$ are the main results of this paper. We state the results for $n=1$ and $n\geq 2$ separately. 

\begin{theorem}(The case of $n\geq 2$) \label{maintheorem} Let $n\geq 2$ 
and $(p_1,p_2,p)$ be a triplet of exponents satisfying $p_1,p_2\geq 2$ and $\frac{1}{p}=\frac{1}{p_1}+\frac{1}{p_2}.$ Then for $\alpha>\alpha_*(p_1,p_2)$ the estimate ~\eqref{mbre} holds with the implicit constant depending only on $\alpha, p_1,p_2$ and $n$.
\end{theorem}

\begin{theorem}\label{dim1}(The case of $n=1$.) Let $\alpha>0$ and $n=1$. Let $1<p_1,p_2<\infty$ and $\frac{1}{p_1}+\frac{1}{p_2}=\frac{1}{p}$. Then the estimate \eqref{mbre} holds for each of the following cases. 
	\begin{enumerate}
		\item $p_1,p_2\geq 2$ and $\alpha>0$.
		\item $1<p_1<2, p_2\geq 2$ and $\alpha>\frac{1}{p_1}-\frac{1}{2}$.
		\item $1<p_2<2, p_1\geq 2$ and $\alpha>\frac{1}{p_2}-\frac{1}{2}$.
		\item \label{case4} $1<p_1,p_2<2$ and $\alpha>\frac{1}{p}-1$. 
	\end{enumerate}
\end{theorem}
\begin{remark}
	Note that in case \eqref{case4} of Theorem~\ref{dim1} we have $p<1$. When $2/3<p<1$ we get that $\frac{1}{p}-1<1/2$ which is interesting as $\alpha=1/2$ is the critical index for the bilinear Bochner-Riesz problem when $n=1$. %In \cite{BGSY} the authors prove a necessary condition on $\alpha$ for $\mathcal{B}_{\alpha}$ to be bounded from $L^{p_1}(\R^n)\times L^{p_2}(\R^n)\rightarrow L^p(\R^n), 1\leq p_1,p_2\leq \infty$ as $\alpha\geq n(\frac{1}{p}-1)-\frac{1}{2}$. In the above theorem when $1<p_1,p_2<2$, \eqref{mbre} holds for $\alpha>1/p-1$. We believe that this range can be improved.
\end{remark}

\begin{center}
	\begin{figure}[]
		\begin{tikzpicture}
		\draw [->, thick, color=gray] (0, 0)--(0, 7);
		\draw [->, thick, color=gray] (0, 0)--(5, 0);
		\foreach \x in {0,...,4} {
			\coordinate (A\x) at (\x,0);
			%\draw[fill=blue] (A\x) circle(1pt);
		}
		\draw (A0) -- (A1) -- (A2) -- (A3)--(A4);
		\node[inner sep=2pt,label=below:$1/4$] (A00) at ($(A0)!2!(A1)$) {};
		\node[inner sep=2pt,label=below:$1/2$] (A10) at ($(A1)!3!(A2)$) {};
		\node[inner sep=2pt,label=below:$1/p$] (A20) at ($(A2)!3!(A3)$) {};
		%\node[inner sep=2pt,label=above:$D$] (A30) at ($(A3)!.5!(A4)$) {};
		%\node[inner sep=2pt,label=above:$E$] (A40) at ($(A4)!.5!(A5)$) {};
		\foreach \y in {0,...,5} {
			\coordinate (A\y) at (0,\y);
			%\draw[fill=blue] (A\y) circle(1pt);
		}
		\draw (A0) -- (A1) -- (A2) -- (A3)--(A4)--(A5);
		\node[inner sep=2pt,label=left:$1/2$] (A00) at ($(A0)!2!(A1)$) {};
		\node[inner sep=2pt,label=left:$1$] (A10) at ($(A1)!3!(A2)$) {};
		\node[inner sep=2pt,label=left:$3/2$] (A20) at ($(A2)!4!(A3)$) {};
		%\node[inner sep=2pt,label=left:$2$] (A20) at ($(A3)!5!(A4)$) {};
		\node[inner sep=2pt,label=left:$\alpha$] (A20) at ($(A4)!3!(A5)$) {};
		\draw [name path=A, thick, dashed, color=gray, label=$\alpha=1-4/p$] (2, 0)--(0, 4);
		\draw [name path=B, thick,dashed, color=gray] (4,0) -- (4,6);
		\draw [name path=C, thick,dashed, color=gray] (0,6) -- (4,6);
		\draw [name path=D, thin,dashed, color=gray, opacity=0.001] (4,0) -- (0,6);
		\node at (2,2) (nodeXj) {$\alpha=1-\frac{4}{p}$};
		\node at (2,2) (nodeXj) {$\alpha=1-\frac{4}{p}$};
		\tikzfillbetween[of=B and C]{gray, opacity=0.1};
		\tikzfillbetween[of=A and D]{gray, opacity=0.1};
		\draw [name path=E, thick,dashed, color=gray] (2,4) -- (4,4);
		\draw [name path=F, thick,dashed, color=gray] (2,4) -- (1,6);
		\draw [name path=G, thick,dashed, color=gray] (1,6) -- (4,6);
		\tikzfillbetween[of=E and G]{gray, opacity=0.2};
		\end{tikzpicture}
		\caption{Range of $\alpha$ and $p$ for $L^p(\R^2)\times L^p(\R^2)\rightarrow L^{p/2}(\R^2)$ boundedness of $\mathcal B_*^{\alpha}$. The region with dark gray color represent the range obtained in~\cite{JL}.}\label{comp}.
	\end{figure}
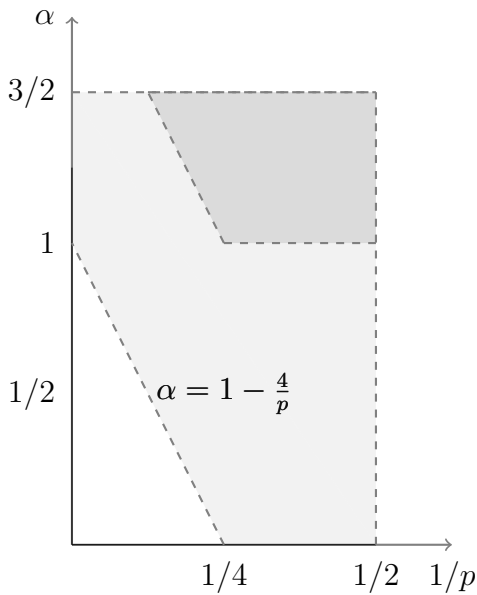
\end{center}
\subsection{Methodology of proof and comparision} 
 We decompose the bilinear multiplier $(1-|\xi|^2-|\eta|^2)_+^{\alpha}$ along the $\xi$ and $\eta$ axes. Note that due to symmetry of the multiplier symbol in~$\xi$ and $\eta$, we need to perform only one of the two decompositions as the other case can be dealt with in a similar fashion. This idea is motivated from the observation that the standard decomposition of the bilinear multipliers $(1-|\xi|^2-|\eta|^2)_+^{\alpha}$ into smooth symbols supported in thin annular regions near the singularity $|\xi|^2+|\eta|^2=1$ requires further decomposition of the smooth symbols into frequency localized multiplier operators using the idea of Littlewood-Paley decompositions. The standard Littlewood-Paley decomposition of these smooth bilinear multipliers poses difficulty when we try to prove estimates for frequency localised operators with supports along the coordinate axes. Due to the curvature of the sphere these Littlewood-Paley Fourier projections overlap along the coordinate axes and it becomes difficult to have a good control on the operators arising in this process. Therefore, in order to overcome this difficulty we decompose the bilinear multiplier $(1-|\xi|^2-|\eta|^2)_+^{\alpha}$ into smooth functions supported in annular regions with respect to variables $\xi$ and $\eta$ separately. This approach along with an identity \eqref{thankstostein} from~\cite{SW} [page 278] yields a very useful decomposition of the bilinear Bochner-Riesz multiplier $(1-|\xi|^2-|\eta|^2)_+^{\alpha}$ into a product of frequency localized Bochner-Riesz square functions, see estimate~\eqref{decomope} for precise detail. 
This approach has a big advantage in obtaining $L^p$ estimates for the bilinear Bochner-Riesz means and its maximal function. When $p_1,p_2\geq 2$ this approach allows us to improve the known range of $\alpha$ significantly for which the estimate~\eqref{mbre} holds. For example, Jeong and Lee [\cite{JL}, Theorem 1.2] has the lower bound on $\alpha$ as $\alpha>\alpha(p_1)+\alpha(p_2)+1$ when $p_1,p_2\geq p_n(s)$ or $p_1=p_2=2$ whereas with our approach we reduce the lower bound to $\alpha>\alpha(p_1)+\alpha(p_2)$. See Figure~ \ref{comp} for a comparision. The region shaded with dark gray color in the figure above represents the range of $p$ versus the index $\alpha$ for the $L^p(\R^2)\times L^p(\R^2)\rightarrow L^{p/2}(\R^2)$ of the maximal function $\mathcal B_*^{\alpha}$ from~\cite{JL}. We extend the estimate to the region shaded with light gray color. In the process, we obtain optimal results for boundedness of bilinear Bochner-Riesz maximal function when $p_1=p_2=2$.

When $n=1$ we get $L^p$ estimates for the bilinear Bochner-Riesz maximal function for a wide range of exponents including the cases when $p_1$ or $p_2$ is less than $2$. Indeed our results for exponents outside the Banach triangle (i.e. $1\leq p_1,p_2,p\leq \infty$) are new even in the case of bilinear Bochner-Riesz operator. Finally, we would like to emphasise that our method is applicable uniformly to all dimensions $n\geq 1$ and provides us with simplified proofs of the existing results. We would like to remark here that this method works best when $p_1,p_2\geq 2$. When either of $ p_1$ and $p_2$ is less than $2$, we need a different approach in order to improve the range of $\alpha$ for the boundedness of Bilinear Bochner Riesz operator and its associated maximal function.
%\vspace{0.2in}

\begin{figure}[H]
\includegraphics[scale=0.75]{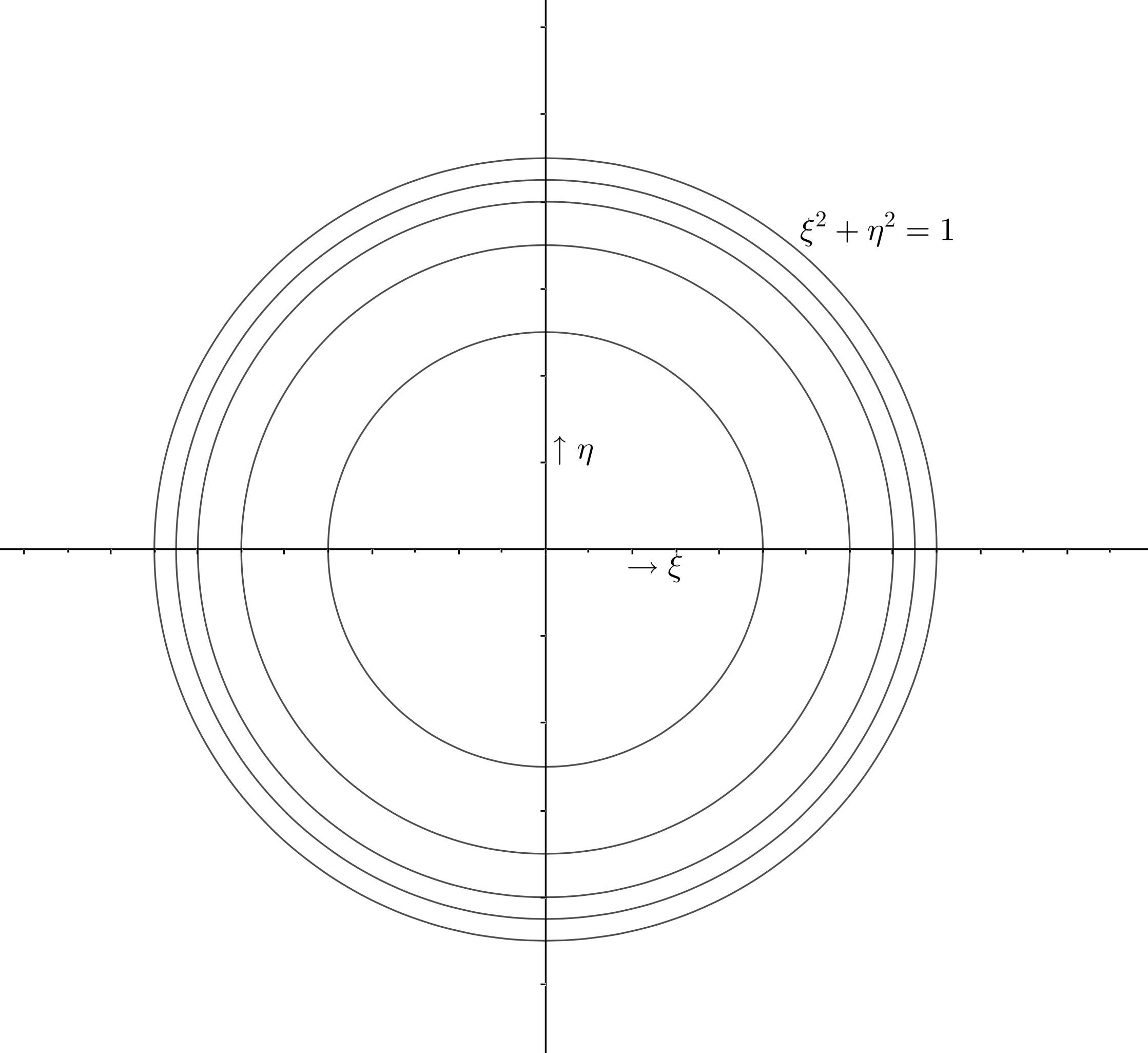}
\caption{This figure represents the standard decomposition of the Bochner- Riesz multiplier in dimension $n=1$ in terms of smooth functions supported on annuli of width comparable to the distance from $|\xi|^2+|\eta|^2=1$.}
\end{figure}
\begin{figure}[H]
\includegraphics[scale=0.85]{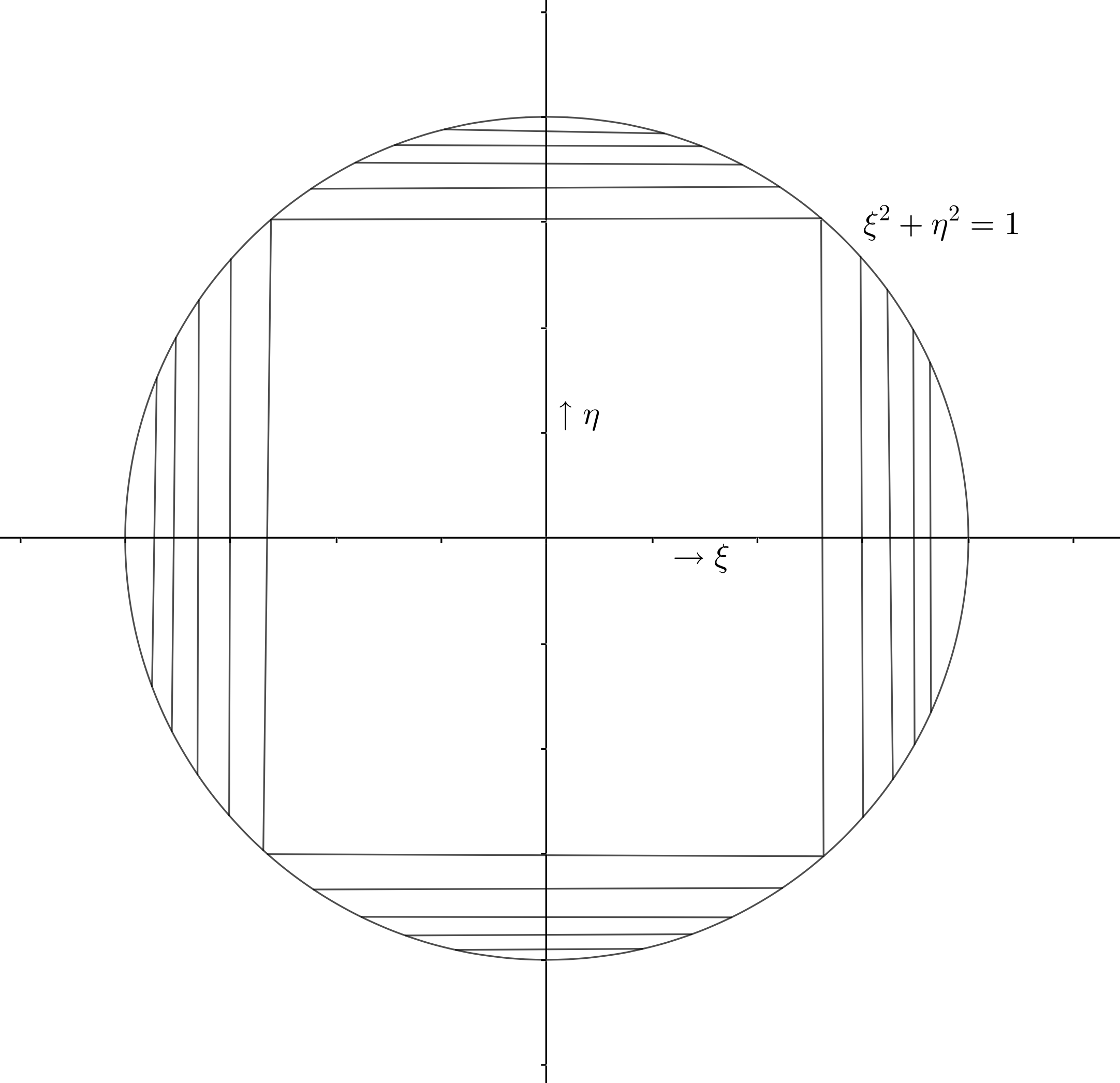}
\caption{This figure represents the decomposition of the Bochner-Riesz multiplier in dimension $n=1$. This involves decomposing along the $\xi$ and $\eta$ axes.}
\end{figure} 

\section{Decomposition of the bilinear Bochner-Riesz multiplier}\label{sec:dec}
Let  $N\geq 20n$ and $I$ be a compact interval of $\R$. Define $C^{N}(I):=\{\psi \in C^{\infty}_c(I): \sup_{x\in I, 0\leq j\leq N}|\frac{d^j}{dx^j}\psi(x)|\leq 1\}$.
Consider the partition of unity for the interval $[0,1],$ i.e.,  for $t\in[0,1]$ consider 
$$1=\sum_{j\geq 2}\psi(2^j(1-t))+\psi_0(t),$$ 
where $\psi\in C^\infty_0[\frac{1}{2},2]$ and $\psi_0\in C_0^{\infty}[-\frac{3}{4},\frac{3}{4}].$
 
%{\includegraphics[scale=1.15]{New decomposition-2.png}}

Let $D=\sup_{x\in [\frac{1}{2},2], 0\leq j\leq N}|\frac{d^j}{dx^j}\psi(x)|$. Then $D^{-1}\psi\in C^{N}([\frac{1}{2},2])$.
		
Setting $t=\frac{|\xi|^2}{R^2}$ in the above we get that  %$\left(1-\frac{|\xi|^2+|\eta|^2}{R^2}\right)_+^{\alpha}$ on both sides we get
\begin{eqnarray*}
	m_R^{\alpha}(\xi,\eta)
	&=&\left(1-\frac{|\xi|^2+|\eta|^2}{R^2}\right)_+^{\alpha}\\
	&=&\sum_{j\geq 2}\psi\left(2^j\left(1-\frac{|\xi|^2}{R^2}\right)\right)\left(1-\frac{|\xi|^2}{R^2}\right)_+^{\alpha}\left(1-\frac{|\eta|^2}{R^2}\left(1-\frac{|\xi|^2}{R^2}\right)^{-1}\right)^{\alpha}_+\\
	&&+\psi_0\left(\frac{|\xi|^2}{R^2}\right)\left(1-\frac{|\xi|^2+|\eta|^2}{R^2}\right)_+^{\alpha}\\
	&=& m^{\alpha}_{1,R}(\xi,\eta)+m^{\alpha}_{0,R}(\xi,\eta),
\end{eqnarray*}
where $$m^{\alpha}_{1,R}(\xi,\eta)=\sum_{j\geq 2}\psi\left(2^j\left(1-\frac{|\xi|^2}{R^2}\right)\right)\left(1-\frac{|\xi|^2}{R^2}\right)_+^{\alpha}\left(1-\frac{|\eta|^2}{R^2}\left(1-\frac{|\xi|^2}{R^2}\right)^{-1}\right)^{\alpha}_+$$ and 
$$m^{\alpha}_{0,R}(\xi,\eta)=\psi_0\left(\frac{|\xi|^2}{R^2}\right)\left(1-\frac{|\xi|^2+|\eta|^2}{R^2}\right)_+^{\alpha}.$$
With the decomposition of the symbols as above we can write 
$$\mathcal{B}^{\alpha}_R(f,g)=\mathcal{B}_{0,R}^{\alpha}(f,g)+\mathcal{B}_{1,R}^{\alpha}(f,g),$$
where \begin{equation}\label{twoparts}\mathcal{B}_{i,R}^{\alpha}(f,g)(x)=\int_{\R^n}\int_{\R^n}m^{\alpha}_{i,R}(\xi,\eta)\hat{f}(\xi)\hat{g}(\eta)e^{2\pi ix.(\xi+\eta)}d\xi d\eta, i=0,1.\end{equation}
\vspace{0.1in}
First, we deal with the bilinear operator $\mathcal{B}_{1,R}^{\alpha}(f,g)$. We further decompose the multiplier as  $$m^{\alpha}_{1,R}(\xi,\eta)=\sum\limits_{j\geq 2}\tilde{m}^{\alpha}_{j,R}(\xi,\eta),$$
where  $$\tilde{m}^{\alpha}_{j,R}(\xi,\eta)=\psi\left(2^j\left(1-\frac{|\xi|^2}{R^2}\right)\right)\left(1-\frac{|\xi|^2}{R^2}\right)_+^{\alpha}\left(1-\frac{|\eta|^2}{R^2}\left(1-\frac{|\xi|^2}{R^2}\right)^{-1}\right)^{\alpha}_+.$$
Therefore we have \begin{equation}\label{sums}\mathcal{B}^{\alpha}_{1,R}(f,g)=\sum_{j\geq 2}T^{\alpha}_{j,R}(f,g)\end{equation} where $T^{\alpha}_{j,R}(f,g)(x)=\int_{\R^n}\int_{\R^n}\tilde{m}^{\alpha}_{j,R}(\xi,\eta)\hat{f}(\xi)\hat{g}(\eta)e^{2\pi ix.(\xi+\eta)}d\xi d\eta$.

Let us denote $\varphi_R(\xi)=\left(1-\frac{|\xi|^2}{R^2}\right)_+.$
	Now using an identity from Stein and Weiss [\cite{SW}, page 278] we have the following relation.
	\begin{equation}\label{thankstostein}\left(1-\frac{|\eta|^2}{R^2\varphi_R(\xi)}\right)^{\alpha}_+=c_{\alpha}R^{-2\alpha}\varphi_R(\xi)^{-\alpha}\int_0^{R}\left(R^2\varphi_R(\xi)-t^2\right)_+^{\beta-1}t^{2\delta+1}\left(1-\frac{|\eta|^2}{t^2}\right)^{\delta}_+dt,\end{equation}
	where $\beta>\frac{1}{2}$, $\delta>-\frac{1}{2}$ and $\beta+ \delta=\alpha.$
	
	Substituting this in the expression of ${\tilde m}^{\alpha}_{j,R}$ we get that 
	\begin{eqnarray}\label{decom}
	&& \nonumber{\tilde m}^{\alpha}_{j,R}(\xi,\eta)\\&=&c_{\alpha} \psi\left(2^j\left(1-\frac{|\xi|^2}{R^2}\right)\right)R^{-2\alpha}\int_0^{R_j}\left(R^2\varphi_R(\xi)-t^2\right)_+^{\beta-1}t^{2\delta+1}\left(1-\frac{|\eta|^2}{t^2}\right)^{\delta}_+dt,
	\end{eqnarray}
	where $R_j=R\sqrt{2^{-j+1}}$. 
	
	Notice that in the equation above due to the support of the function  $\psi\left(2^j\left(1-\frac{|\xi|^2}{R^2}\right)\right)$ the upper limit of $t$ in the integral is $R_j$.
	Define \begin{eqnarray}\label{B_j}B_{j,\beta}^{R,t}f(x)=\int_{\R^n}\hat{f}(\xi)\psi\left(2^j\left(1-\frac{|\xi|^2}{R^2}\right)\right)\left(R^2\varphi_R(\xi)-t^2\right)_+^{\beta-1}e^{2\pi i x.\xi}d\xi
	\end{eqnarray}
	and \begin{eqnarray}\label{BR}B_t^{\delta}g(x)=\int_{\R^n}\hat{g}(\eta)\left(1-\frac{|\eta|^2}{t^2}\right)^{\delta}_+e^{2\pi ix.\eta} d\eta.
	\end{eqnarray}
	This yields the following decompostion of the bilinear operator $T^{\alpha}_{j,R}$ associated with the symbol ${\tilde m}^{\alpha}_{j,R}(\xi,\eta)$.
	\begin{eqnarray} T^{\alpha}_{j,R}(f,g)(x)&=&\nonumber  \int_{\R^n}\int_{\R^n}\tilde{m}^{\alpha}_{j,R}(\xi,\eta)\hat{f}(\xi)\hat{g}(\eta)e^{2\pi ix.(\xi+\eta)}d\xi d\eta\\
		&=&\label{decomope} c_{\alpha}\int_0^{R_j}B_{j,\beta}^{R,t}f(x)B_t^{\delta}g(x)t^{2\delta+1}dt,
		\end{eqnarray}
		where $\alpha=\beta+\delta$.
		
		We will show that the decomposition of the operator as above yields new estimates for the bilinear Bochner-Riesz maximal function. Before going into the proofs of our main results we discuss some useful results concerning the Bochner-Riesz square function. 
\section{A brief discussion on the Bochner-Riesz Square function}\label{results:used}
%In this section we discuss some known results ceocerning square function for the classical Bochner-Riesz means. These results will be used in proving the main results of this paper. 

The Bochner-Riesz square function $G^{\alpha}$ is defined by 
\begin{eqnarray*}
	G^{\alpha}(f)(x):=\left(\int_0^{\infty}|\frac{\partial}{\partial t}B_t^{\alpha+1}f(x)|^2 tdt \right)^{1/2}=\left(\int_{0}^{\infty}|\mathcal K^{\alpha}_t\ast f(x)|^2\frac{dt}{t}\right)^{1/2},
\end{eqnarray*}
where $\widehat{B_t^{\alpha}f}(\xi)=\left(1-\frac{|\xi|^2}{t^2}\right)^{\alpha}_+\hat f(\xi)$ and $\widehat{\mathcal K^{\alpha}_t}(\xi)=2(\delta+1)\frac{|\xi|^2}{t^2}
(1-\frac{|\xi|^2}{t^2})^{\alpha}_+.$ 

The square function $G^{\alpha}$ was introduced by Stein~\cite{St}. The $L^p$ boundedness \begin{eqnarray}\label{squarefun}
\|G^{\alpha}(f)\|_{L^p(\R^n)} \lesssim \|f\|_{L^p(\R^n)}
\end{eqnarray}
 of the square function have been studied extensively in the literature. It has various applications, in particular, it plays an improtant role in the study of maximal Bochner-Riesz functions. We refer the reader to~\cite{St, Sun,KanSun, Car, Ch, See, LRS2, LRS, Leesqr} and references there in for details. 
 
 Due to the derivative with respect to $t$ the operator $G^{\alpha}$ essentially behaves like the Bochner-Riesz operator $B^{\alpha}$ of order $\alpha$. The Plancherel theorem immediately yields that $G^{\alpha}$ is bounded on $L^2(\R^n)$ provided $\alpha>-\frac{1}{2}$, see~\cite{St}. It is conjectured that for $1<p\leq 2$, the estimate~(\ref{squarefun}) holds if, and only if $\alpha>n(\frac{1}{p}-\frac{1}{2})-\frac{1}{2}.$ Whereas for the range $p>2$ the behaviour of $G^{\alpha}$ is different and it is conjectured that for $p>2$, the estimate~(\ref{squarefun}) holds if, and only if $\alpha>\alpha(p)-\frac{1}{2}.$ 
 
The conjecture for the range $1<p\leq 2$ has been resolved. We know that the estimate~(\ref{squarefun}) holds if, and only if $\alpha>n(\frac{1}{p}-\frac{1}{2})-\frac{1}{2}$ for $1<p\leq 2$ and $n\geq 1$, see~\cite{Sun, KanSun, LRS2} for detail. Further, in dimensions $n=1,2$, the conjecture has been proved to hold for the range $p\geq 2$ , see~\cite{KanSun} and \cite{Car} for the case of $n=1$ and $n=2$ respectively. When $n\geq 3$ and $p>2$ the sufficient part of the conjecture is not known completely. There are many interesting developments on the conjecture and we refer to~\cite{Ch, See, LRS2, LRS, Leesqr} for more detail. The most recent develpoment in this direction is due to ~Lee~\cite{Leesqr}, where he proved the following result. 
\begin{theorem}\cite{Leesqr}\label{square:Lee}
Let $n\geq 2$. The estimate~(\ref{squarefun}) holds for $p\geq \min\{\mathfrak p_n,\frac{2(n+2)}{n}\}$ and $\alpha>n(\frac{1}{2}-\frac{1}{p})-1$.
\end{theorem}
The $L^p$ estimates for the square function mentioned as above will be used in order to prove our results. Indeed, we shall need $L^p$ boundedness of certain local variants of the square function. We consider the following setting. 

For $\psi\in C^{N+1}([\frac{1}{2},2])$ and $0<\nu <\frac{1}{16}$ define $$B^{\psi}_{\nu,t}f(x)=\int_{\R^n}\psi\left(\nu^{-1}\left(1-\frac{|\xi|^2}{t^2}\right)\right)\hat{f}(\xi)e^{2\pi ix.\xi}d\xi.$$

Also, for $\phi\in C^N([-1,1])$ and $0<\nu <\frac{1}{16}$ consider the  operator  $$\phi\left(\frac{t^2-|D|^2}{\nu}\right)f(x)=\int_{\R^n}\phi\left(\frac{t^2-|\xi|^2}{\nu}\right)\hat{f}(\xi)e^{2\pi ix.\xi}d\xi.$$
Jeong, Lee and Vargas~[Lemma $2.6$ \cite{JLV}] proved that for any $\epsilon>0$, there exists $N\geq 1$ such that for $\phi\in C^N([-1,1])$ the following holds \begin{equation}\label{localised1}\left\|\left(\int_{1/2}^{2}\left|\phi\left(\frac{t^2-|D|^2}{\nu}\right)f(\cdot)\right|^2 dt\right)^{1/2}\right\|_{L^p(\R^n)}\lesssim_{\epsilon,N} \nu^{(1/2-\alpha(p))}\nu^{-\epsilon}\|f\|_{L^p(\R^n)}
\end{equation}
 when $p>\mathfrak p_n$. 
	
We need the following version of the result stated above.
\begin{lemma}\label{localisedversion} Let $n\geq 2, 0<\nu<\frac{1}{16}$ and $\epsilon>0$. Then for $p\geq \mathfrak p_n$ and $p=2,$ there exists $N\geq 1$ such that for all $\psi\in C^{N+1}([1/2,2])$ the following holds 
\begin{equation}\label{localised2}\left\|\left(\int_{0}^{\infty}\left|B^{\psi}_{\nu,t}f(\cdot)\right|^2 \frac{dt}{t}\right)^{1/2}\right\|_{L^p(\R^n)}\lesssim_{\epsilon,N} \nu^{(1/2-\alpha(p))}\nu^{-\epsilon}\|f\|_{L^p(\R^n)}
\end{equation}
where the implicit constant depends on $\epsilon$ and $N$.
\end{lemma}
In Lemma~\ref{localisedversion} we can take the same value of $N$ as in Lemma $2.6$ in \cite{JLV}. The proof of the lemma above may be completed using Lemma $2.6$ in \cite{JLV} with minor modifications. We present the proof here for an easy reference and completeness.
	
\begin{proof} Note that when $p=2$ the inequality \eqref{localised2} is a consequence of Plancherel theorem. 
		
	Next, let $\psi_r(x):=r^{-(N+1)}\psi(rx)$ where $r>2$ is fixed such that $\psi_r\in C^{N+1}([-1,1])$. First, consider the operator $B^{\psi_r}_{\nu,t}$. 
	
	When $p>2$ we know that by a standard Littlewood Paley decomposition, it is enough to prove the inequality \eqref{localised2} for $\frac{1}{2}<t<2$. Consider 
	\begin{eqnarray*}B^{\psi_r}_{\nu,t}f(x)&=&B^{\psi_r}_{\nu,t}f(x)-\psi_r\left(\frac{t^2-|D|^2}{4^{-1}\nu}\right)f(x)+\psi_r\left(\frac{t^2-|D|^2}{4^{-1}\nu}\right)f(x)\\
		&=&\psi_r\left(\frac{t^2-|D|^2}{4^{-1}\nu}\right)f(x)+\int_{1/4}^{t^2}	\widetilde{\psi_r}\left(\frac{t^2-|D|^2}{\nu s}\right)f(x)\frac{ds}{s},	\end{eqnarray*}
	where $\widetilde{\psi_r}(x)=r^{-(N+1)}rx\psi'(rx)$.
	
	It is easy to verify that $\widetilde{\psi_r}$ belongs to $C^N([-1,1])$. Using Minkowski's integral inequality and triangle inequality in the identity above, we get that 
\begin{eqnarray*}
\left\|\left(\int_{1/2}^{2}\left|B^{\psi_r}_{\nu,t}f(\cdot)\right|^2 \frac{dt}{t}\right)^{1/2}\right\|_{L^p(\R^n)} 
&\lesssim&\left\|\left(\int_{1/2}^{2}\left|\psi_r\left(\frac{t^2-|D|^2}{4^{-1}\nu}\right)f(\cdot)\right|^2 dt\right)^{1/2}\right\|_{L^p(\R^n)}\\
&&+\int_{1/4}^{4}\left\|\left(\int_{1/2}^{2}\left|\psi_r\left(\frac{t^2-|D|^2}{\nu s}\right)f(\cdot)\right|^2 dt\right)^{1/2}\right\|_{L^p(\R^n)}	\frac{ds}{s}.
\end{eqnarray*}
From \eqref{localised1} we know that for all $s\in[1/4,4]$ the following holds.  $$\left\|\left(\int_{1/2}^{2}\left|\psi_r\left(\frac{t^2-|D|^2}{\nu s}\right)f(\cdot)\right|^2 dt\right)^{1/2}\right\|_{L^p(\R^n)}\lesssim_{\epsilon,N}(\nu s)^{(1/2-\alpha(p))}\nu^{-\epsilon}\|f\|_{L^p(\R^n)}.$$
	Therefore, we get that  $$\left\|\left(\int_{1/2}^{2}\left|B^{\psi_r}_{\nu,t}f(\cdot)\right|^2 \frac{dt}{t}\right)^{1/2}\right\|_{L^p(\R^n)}\lesssim_{\epsilon, N} \nu^{(1/2-\alpha(p))}\nu^{-\epsilon}\|f\|_{L^p(\R^n)}$$
	for $p> \mathfrak p_n$.
	
	The standard interpolation argument between $p=2$ and $p>\mathfrak p_n$ yields that the inequality as above holds for $p=\mathfrak p_n$.
	
	Further, it is easy to verify that $r^{N+1}B^{\psi_r}_{\nu,t}f(x)=B^{\psi}_{r^{-1}\nu,t}f(x)$. Note that for a fixed $r>2$ we can choose $\nu$ small enough so that $r^{-1}\nu<1/16$.
	This completes the proof of the inequality \eqref{localised2}.
	\end{proof}
	Next, we consider the maximal function $\sup_{t>0}|B^\psi_{\nu,t}f(x)|$ associated with the frequency localised Bochner-Riesz operators. The square function estimate proved in Lemma \ref{localisedversion} gives us the following estimate 
	\begin{equation}\label{localisedmaximal}
	\|\sup_{t>0}|B^{\psi}_{\nu,t}f|\|_{L^p(\R^n)}\lesssim_{\epsilon, N}\nu^{-\epsilon}	\nu^{\alpha(p)}\|f\|_{L^p(\R^n)}
	\end{equation}
for $p=2$ and $p\geq \mathfrak p_n.$

Note that if $m\in C^{\infty}_0([1/2,2])$ is such that $A^{-1}m\in C^N([1/2,2])$ for some $A>0$, then from \eqref{localisedmaximal} we get that \begin{equation}\label{localisedmaximal2}
 	\|\sup_{t>0}|B^{m}_{\nu,t}f|\|_{L^p(\R^n)}\lesssim_{\epsilon,N}A\nu^{-\epsilon}	\nu^{\alpha(p)}\|f\|_{L^p(\R^n)}
 \end{equation}
holds for the same range of $p$ as in \eqref{localisedmaximal}.
	
	Let $h$ be a smooth function supported on a compact interval $I$ such that   $C^{-1}h\in C^N(I)$ for some $C>0$. In what follows we shall use the statement that $h_k(x)=x^kh(x),k\geq 1$ and $h(x)$ behave similarly. This would mean that the corresponding square function and the maximal function for the operator $B^{h_k}_{\nu,t}$ defined above are bounded for the same range of $p$ as that of $B^{h}_{\nu,t}$ with the same constants except for an extra factor of $k!$. 

Next, we make use of a technique from~[\cite{SW}, page 277-278] to derive the $L^p$ boundedness of the operator $f\rightarrow \sup_{R>0}\left(R^{-1}\int_0^{R}|B_t^{\delta}f(\cdot)|^2dt\right)^{1/2}$. This idea gives us the following estimate. 
\begin{lemma}\label{steinlemma} The operator 
		\begin{equation*}\label{steinsquare}f\rightarrow \sup_{R>0}\left(R^{-1}\int_0^{R}|B_t^{\delta}f(\cdot)|^2dt\right)^{1/2}\end{equation*} is bounded on $L^p(\R^n)$ for $p=2$ and $p\geq \mathfrak p_n$ where $\delta>\alpha(p)-1/2$.
		\end{lemma}
		\begin{proof} The case of $p=2$ is easy as usual. When $p>2$ we write $$B_t^{\delta}f=\sum_{k=1}^{d}\left(B_t^{\delta+k-1}f-B_t^{\delta+k}f\right) + B_t^{\delta+d}f,$$
		where $d$ is chosen so that $\delta+d>\frac{n-1}{2}$.
		
		This implies that  $$\left(\int_0^{R}|B_t^{\delta}f(x)|^2dt\right)^{1/2}\leq \sum_{k=1}^{d}\left(\int_0^{R}|B_t^{\delta+k}f(x)-B_t^{\delta+k-1}f(x)|^2dt\right)^{1/2} + \left(\int_0^{R}|B_t^{\delta+d}f(x)|^2dt\right)^{1/2}.$$
		Since $\delta+d >\frac{n-1}{2}$ the convolution kernel of $B_t^{\delta}$ is an integrable function and consequently $\sup_{t>0}|B_t^{\delta}f|\leq c(\delta+d,n)\mathcal{M}f(x)$,
		where $\mathcal{M}(\cdot)$ is the Hardy-Littlewood Maximal function.
		Therefore $\sup_{R>0}\left(R^{-1}\int_0^{R}|B_t^{\delta+d}f(x)|^2dt\right)^{1/2}$ is bounded on $L^p(\R^n),1<p\leq \infty$.
	
		Next, consider the term $$\sup_{R>0}\left(R^{-1}\int_0^{R}|B_t^{\delta+k}f(x)-B_t^{\delta+k-1}f(x)|^2dt\right)^{1/2}$$ for a fixed $1\leq k\leq d$. Clearly the above can be dominated by $$\left(\int_0^{\infty}|B_t^{\delta+k}f(x)-B_t^{\delta+k-1}f(x)|^2 t^{-1}dt\right)^{1/2}$$
		which is nothing but the square function for Bochner-Riesz means of order $\delta+k-1$. Invoking Theorem~\ref{square:Lee} from \cite{Leesqr} we get the desired $L^p(\R^n)$ estimates for $p\geq \mathfrak p_n$ when $\delta>\alpha(p)-1/2$ for all $k\geq 1$. This completes the proof. 
		\end{proof}	
%%%%%%%%%%%%%%%%%%%%%%%%%
%%%%%%%%%%%%%%%%%%%%%%%%%%
\section{Proof of Theorem~\ref{maintheorem}: Boundedness of bilinear Bochner-Riesz maximal function}\label{proof:max}
%%%%%%%%%%%%%%%%%%%%%%%%%
This section is devoted to proving Theorem~\ref{maintheorem}. We shall prove the estimates for $p_j=2$ or $p_i\geq \mathfrak p_n,~i,j=1,2$ and the remaining cases in Theorem~\ref{maintheorem} will be obtained by interpolation arguments. 

Recall that from section~\ref{sec:dec} we have the following decomposition of the bilinear Bochner-Riesz operator $\mathcal B^{\alpha}_R$. 
$$\mathcal{B}^{\alpha}_R(f,g)=\mathcal{B}_{0,R}^{\alpha}(f,g)+\mathcal{B}_{1,R}^{\alpha}(f,g).$$ In order to prove Theorem~\ref{maintheorem} it is enough to prove the desired estimates for the bilinear maximal functions 
$$\mathcal B^{\alpha}_{i,*}(f,g)(x)=\sup_{R>0}|\mathcal{B}^{\alpha}_{i,R}(f,g)(x)|,~i=0,1.$$
We will deal with both the maximal functions separately. Let us first consider the case of $i=1$. 
\subsection{Boundedness of the bilinear maximal function $\mathcal B^{\alpha}_{1,*}(f,g)$}
Recall that from section~\ref{sec:dec} we have the following decomposition  \begin{equation*}\mathcal{B}^{\alpha}_{1,R}(f,g)(x)=\sum_{j\geq 2}T^{\alpha}_{j,R}(f,g)(x),
\end{equation*} 
where $T^{\alpha}_{j,R}(f,g)(x)=\int_{\R^n}\int_{\R^n}\tilde{m}^{\alpha}_{j,R}(\xi,\eta)\hat{f}(\xi)\hat{g}(\eta)e^{2\pi ix.(\xi+\eta)}d\xi d\eta$. Therefore, it is enough to consider the maximal function $$T^{\alpha}_{j,*}(f,g)(x)=\sup_{R>0}|T^{\alpha}_{j,R}(f,g)(x)|~\text{for~ each~} j\geq 2.$$

Applying Cauchy Schwartz inequality we get that 
$$|T_{j,R}^{\alpha}(f,g)(x)|\leq c_{\alpha} \left(\int_0^{R_j}|B_{j,\beta}^{R,t}f(x)t^{2\delta+1}|^2 dt\right)^{1/2}\left(\int_0^{R_j}|B_t^{\delta}g(x)|^2dt\right)^{1/2}.$$
Making a change of variable $t\rightarrow Rt$ in the integral $\left(\int_0^{R_j}|B_{j,\beta}^{R,t}f(x)t^{2\delta+1}|^2 dt\right)^{1/2}$ we get that $$\int_0^{R_j}|B_{j,\beta}^{R,t}f(x)t^{2\delta+1}|^2 dt=R^{4\alpha-1}\int_0^{\sqrt{2^{-j+1}}}|{S}_{j,\beta}^{R,t}f(x)t^{2\delta+1}|^2 dt,$$
where $${S}_{j,\beta}^{R,t}f(x)=\int_{\R^n}\psi\left(2^j\left(1-\frac{|\xi|^2}{R^2}\right)\right)\left(1-\frac{|\xi|^2}{R^2}-t^2\right)_+^{\beta-1}\hat{f}(\xi)e^{2\pi ix.\xi}d\xi.$$ 
Finally, we get that  
\begin{eqnarray}\label{reduction}
T^{\alpha}_{j,*}(f,g)(x)
&\leq&  2^{-j/4}\sup_{R>0}\left(\int_0^{\sqrt{2^{-j+1}}}|{S}_{j,\beta}^{R,t}f(x)t^{2\delta+1}|^2dt\right)^{1/2}\\
&&\nonumber \left(\sup_{R>0}R_j^{-1}\int_0^{R_j}|B_t^{\delta}g(x)|^2dt\right)^{1/2}.
\end{eqnarray}
We have the following result for the maximal function involving the operator ${S}_{j,\beta}^{R,t}$ in the inequality above. 
\begin{theorem}\label{maxsquare}
Let $n\geq 2.$ For $p\geq \mathfrak p_n$ or $p=2$ and $\beta> \alpha(p)+1/2,$  the following estimate holds 
\begin{eqnarray}\label{reducedoperator} \left\|\sup_{R>0}\left(\int_{0}^{\sqrt{2^{-j+1}}}|{S}_{j,\beta}^{R,t}f(\cdot)t^{2\delta+1}|^2 dt\right)^{1/2}\right\|_{L^p(\R^n)}\lesssim 2^{j(\alpha(p)-\alpha+\frac{1}{4}+\epsilon)}\|f\|_{L^p(\R^n)}.
\end{eqnarray}
\end{theorem}
We postpone the proof of Theorem~\ref{maxsquare} to the next subsection. Assuming Theorem~\ref{maxsquare} we complete the estimate for the maximal function $\mathcal B^{\alpha}_{1,*}$.

Note that the H\"{o}lder inequality in the estimate \eqref{reduction} yields that 
%After applying Holder's inequality in \eqref{reduction} we get
\begin{eqnarray*} &&\|T^{\alpha}_{j,*}(f,g)\|_{L^p(\R^n)} \\&\lesssim& 2^{-j/4}\left\|\sup_{R>0}\left(\int_0^{\sqrt{2^{-j+1}}}|{S}_{j,\beta}^{R,t}f(\cdot)t^{2\delta+1}|^2 dt\right)^{1/2}\right\|_{L^{p_1}(\R^n)}\left\|\left(\sup_{R>0}R_j^{-1}\int_0^{R_j}|B_t^{\delta}g(\cdot)|^2dt\right)^{1/2}\right\|_{L^{p_2}(\R^n)},
	\end{eqnarray*}
where $\delta+\beta=\alpha.$

Invoking the estimates from Theorem~\ref{maxsquare} and Lemma \ref{steinlemma} we get that 
\begin{eqnarray}\label{better} \|T^{\alpha}_{j,*}(f,g)\|_{L^p(\R^n)} & \lesssim_{n,\alpha,p_1,p_2}&2^{-j(\alpha-\alpha(p_1))}\|f\|_{L^{p_1}(\R^n)}\|g\|_{L^{p_2}(\R^n)}
\end{eqnarray} 
where $\beta>\alpha(p_1)+1/2$ for $p_1=2$ or $p_1\geq \mathfrak p_n$ and $\delta>\alpha(p_2)-1/2$ for $p_2=2$ or $p_2\geq \mathfrak p_n$, which is same as saying that $\alpha>\alpha(p_1)+\alpha(p_2)=\alpha_*(p_1,p_2)$ when $p_i=2$ or $p_i\geq \mathfrak p_n$ for $i=1,2$. 

The decomposition \eqref{sums} along with the estimate above implies that for $\alpha>\alpha_*(p_1,p_2)$ we have 
\begin{eqnarray} \label{b1*}\|\mathcal B^{\alpha}_{1,*}(f,g)\|_{L^p(\R^n)} & \lesssim_{n,\alpha,p_1,p_2}&\|f\|_{L^{p_1}(\R^n)}\|g\|_{L^{p_2}(\R^n)}.
\end{eqnarray} 
This proves the desired estimate for the maximal function $\mathcal B^{\alpha}_{1,*}(f,g)$ under the assumption that Theorem~\ref{maxsquare} holds.  
%%%%%%%%%%%%%%%%%%%%%%%%
\subsection{Proof of Theorem~\ref{maxsquare}\label{sec:maxsquare}}
%%%%%%%%%%%%%%%%%%%%%%%%%%
We set $\beta=\gamma+1$ in ${S}_{j,\beta}^{R,t}f$ for convenience. 

Note that when $t^2\in[0,2^{-j-1})$ we can rewrite the function  $$\psi\left(2^j\left(1-\frac{|\xi|^2}{R^2}\right)\right)\left(1-\frac{|\xi|^2}{R^2}-t^2\right)_+^{\gamma}=\psi\left(2^j\left(1-\frac{|\xi|^2}{R^2}\right)\right)\left(1-\frac{|\xi|^2}{R^2}-t^2\right)^{\gamma}.$$ 
Moreover, in this range of $t,$ the multiplier has no singularity. Indeed it behaves like $2^{-j\gamma}$ times a smooth function supported in the annulus of width of the order $R\sqrt{2^{-j}}$. 

With this observation in mind, we split the interval $[0,2^{-j+1}]$, the range of $t^2$ for the multiplier $\left(1-\frac{|\xi|^2}{R^2}-t^2\right)_+^{\beta-1}$, into two subintervals $[0,2^{-j-1-\epsilon_0}]$ and $[2^{-j-1-\epsilon_0},2^{-j+1}]$ and deal with the corresponding operators separately. Here we have taken $0<\epsilon_0<<1$ to be a fixed number.  
\subsection*{Case I: When $t\in[0,\sqrt{2^{-j-1-\epsilon_0}}]$}
%\begin{lemma}\label{easypart} 
For $\gamma>-1$ and $p=2$ or $p\geq \mathfrak p_n$, we have 
\begin{eqnarray}\label{easypart}
\left\|\sup_{R>0}\left(\int_0^{\sqrt{2^{-j-1-\epsilon_0}}}|{S}_{j,\gamma+1}^{R,t}f(x)t^{2\delta+1}|^2 dt\right)^{1/2}\right\|_{L^p(\R^n)}&\lesssim & 2^{-j(\delta +\gamma+3/4-\alpha(p))}\|f\|_{L^p(\R^n)}.
\end{eqnarray}
%\end{lemma}

%By interpolation argument we get that  $$\left\|\sup_{R>0}\left(\int_0^{\sqrt{2^{-j-1-\epsilon_0}}}|{S}_{j,\gamma+1}^{R,t}f(x)t^{2\delta+1}|^2 dt\right)^{1/2}\right\|_{L^p(\R^n)}\lesssim 2^{-j(\delta+1/2)}2^{-j\gamma}2^{-j/4}2^{j\alpha(p_1)}\|f\|_{L^p(\R^n)}.$$
We first consider the range $-1<\gamma<0.$ Set  $\gamma=-\rho$, where $\rho>0$.
%Look at the function $\left(1-\frac{|\xi|^2}{R^2}-t^2 \right)^{-\rho}.$
By Taylor's expansion we can write 
\begin{equation}\label{taylorexansion}\left(1-\frac{|\xi|^2}{R^2}\right)^{-\rho}\left(1-\frac{t^2}{1-\frac{|\xi|^2}{R^2}} \right)^{-\rho}=2^{j\rho}\left(2^{j}\left(1-\frac{|\xi|^2}{R^2}\right)\right)^{-\rho}\sum_{k\geq 0}\frac{\Gamma(\rho+k)}{\Gamma(\rho)k!}\left(\frac{2^jt^2}{2^j(1-\frac{|\xi|^2}{R^2})}\right)^k.\end{equation}
Observe that the series above converges as $\frac{t^2}{1-\frac{|\xi|^2}{R^2}}\leq 2^{-j-1-\epsilon_0}2^{j+1}<1.$
Therefore, using \eqref{taylorexansion} we get that 
\begin{eqnarray}\label{s_j}{S}_{j,\gamma+1}^{R,t}f&=&2^{j\rho}\sum_{k\geq0}\frac{\Gamma(\rho+k)}{\Gamma(\rho)k!}\left(2^jt^2\right)^kB^{\psi^k}_{2^{-j},R}f.
\end{eqnarray}
Denote $\psi^k(x):= x^{-k-\rho}\psi(x)$. Observe that $\psi^k\in C^{\infty}_0([1/2,2])$ and it satisfies the estimate
 $$\sup_{x\in[1/2,2],0\leq l\leq N}\left|\frac{d^l\psi^k}{dx^l}\right|\leq C(\rho)2^{N+k} k^{N+1}$$
 for $N\geq 20n$. 
Therefore, the corresponding maximal function is  bounded on $L^p(\R^n)$ with an additional factor of  $ 2^{N+k} k^{N}2^{j\alpha(p)}$, see the estimate~\eqref{localisedmaximal2} for detail. More precisely, we get that  $$\|\sup_{R>0}|B^{\psi^k}_{2^{-j},R}f|\|_{L^p(\R^n)}\lesssim_{N,p_1}2^{N+k} k^{N}2^{j\alpha(p)}\|f\|_{L^p(\R^n)}$$
where $p=2$ or $p \geq \mathfrak p_n$.

Using the Minkowski's integral inequality and the fact that $2^jt^2\leq 2^{-1-\epsilon_0}$ we have 
\vspace{0.4in} 
\begin{eqnarray*}
\left\|\sup_{R>0}\left(\int_0^{\sqrt{2^{-j-1-\epsilon_0}}}|{S}_{j,\gamma+1}^{R,t}f(\cdot)t^{2\delta+1}|^2 dt\right)^{1/2}\right\|_{L^p(\R^n)}&\lesssim&  \left(\int_0^{\sqrt{2^{-j-1-\epsilon_0}}}\|\sup_{R>0}|{S}_{j,\gamma+1}^{R,t}f|\|_{L^p(\R^n)}^2t^{4\delta +2}dt\right)^{1/2}\\
&\lesssim & 2^{-j(\delta+1/2)}2^{-j\gamma}2^{-j/4}2^{j\alpha(p)}\\&&\times\sum_{k\geq 0}\frac{\Gamma(\rho+k)}{\Gamma(\rho)k!}2^{N+k} k^{N}2^{-k-\epsilon_0 k} \|f\|_{L^p(\R^n)}.
\end{eqnarray*}
By the asymptotic formula of Gamma function we know that $\frac{\Gamma(\rho+k)}{\Gamma(\rho)k!}\approx k^{\rho}$
and hence the series in the expression above converges. This gives us the desired estimate. 
 
Note that when $\gamma\geq 0,$ we prove the desired estimates for $\gamma$ to be non-negative integers.
For $\gamma=0,$	we know that 
$$f\rightarrow \sup_{R>0}|{S}_{j,\gamma+1}^{R,t}f|$$
is bounded on $L^{p}(\R^n)$ with the operator norm bounded by a constant multiple of  $2^{j\alpha(p)}$ for $p=2$ or $p\geq \mathfrak p_n$ (see \cite{Lee} for detail). Therefore, we get that  $$\left\|\sup_{R>0}\left(\int_0^{\sqrt{2^{-j-1-\epsilon_0}}}|{S}_{j,1}^{R,t}f(x)t^{2\delta+1}|^2 dt\right)^{1/2}\right\|_{L^p(\R^n)}\lesssim 2^{-j(\delta+1/2)}2^{-j/4}2^{j\alpha(p)}\|f\|_{L^p(\R^n)}.$$
Note that the term $2^{-j/4}$ occurs because of integration in the variable $t$.

Next, when $\gamma=1$ we write  \begin{align}\label{gamma}\psi\left(2^j\left(1-\frac{|\xi|^2}{R^2}\right)\right)\left(1-\frac{|\xi|^2}{R^2}-t^2\right)=2^{-j}\left(2^{j}\left(1-\frac{|\xi|^2}{R^2}\right)\psi\left(2^j\left(1-\frac{|\xi|^2}{R^2}\right)\right)\right)\blue{-} t^2\psi\left(2^j\left(1-\frac{|\xi|^2}{R^2}\right)\right).
\end{align}

Since the function  $\phi\left(1-\frac{|\xi|^2}{R^2}\right):=2^{j}\left(1-\frac{|\xi|^2}{R^2}\right)\psi\left(2^j\left(1-\frac{|\xi|^2}{R^2}\right)\right)$ behaves like 
%\red {We can mention it in the beginning what do we mean by ``behaves like'' as we are using this phrase at many places.} 
$\psi\left(2^j\left(1-\frac{|\xi|^2}{R^2}\right)\right)$, the identity as above gives us that 
$$|{S}_{j,2}^{R,t}f(x)|\leq 2^{-j}|B^{\phi}_{2^{-j},R}f(x)|+ t^2|B^{\psi}_{2^{-j},R}f(x)|.$$
Using the identity above  along with the estimate~\eqref{localisedmaximal2} we get that  $$\left\|\sup_{R>0}\left(\int_0^{\sqrt{2^{-j-1-\epsilon_0}}}|{S}_{j,2}^{R,t}f(x)t^{2\delta+1}|^2 dt\right)^{1/2}\right\|_{L^p(\R^n)}\lesssim 2^{-j(\delta+1/2)}2^{-j}2^{-j/4}2^{j\alpha(p_1)}\|f\|_{L^p(\R^n)}.$$
 When $0<\gamma<1$, we write $\gamma=\zeta+1$ for $\zeta\in (-1,0)$.
 
 We can write  $$\psi\left(2^j\left(1-\frac{|\xi|^2}{R^2}\right)\right)\left(1-\frac{|\xi|^2}{R^2}-t^2\right)^{\gamma}=2^{-j}\Psi_1\left(\frac{|\xi|^2}{R^2}\right)\left(1-\frac{|\xi|^2}{R^2}-t^2\right)^{\zeta}-t^2 \Psi_2\left(\frac{|\xi|^2}{R^2}\right)\left(1-\frac{|\xi|^2}{R^2}-t^2\right)^{\zeta},$$
 
 where $\Psi_1\left(\frac{|\xi|^2}{R^2}\right)=2^j\left(1-\frac{|\xi|^2}{R^2}\right) \psi\left(2^j\left(1-\frac{|\xi|^2}{R^2}\right)\right)$ and $\Psi_2\left(\frac{|\xi|^2}{R^2}\right)=\psi\left(2^j\left(1-\frac{|\xi|^2}{R^2}\right)\right)$
 %Now \begin{align*}\psi\left(2^j\left(1-\frac{|\xi|^2}{R^2}\right)\right)\left(1-\frac{|\xi|^2}{R^2}-t^2\right)^{\gamma}=\left(1-\frac{|\xi|^2}{R^2}\right) \psi\left(2^j\left(1-\frac{|\xi|^2}{R^2}\right)\right)\left(1-\frac{|\xi|^2}{R^2}-t^2\right)^{\zeta}\\=2^{-j}\left(\psi_1\left(2^j\left(1-\frac{|\xi|^2}{R^2}\right)\right)\left(1-\frac{|\xi|^2}{R^2}-t^2\right)^{\zeta}\right)- t^2\psi\left(2^j\left(1-\frac{|\xi|^2}{R^2}\right)\right)\left(1-\frac{|\xi|^2}{R^2}-t^2\right)^{\zeta},\end{align*}
 %where $\psi_1\left(2^j\left(1-\frac{|\xi|^2}{R^2}\right)\right)=2^{j}\left(1-\frac{|\xi|^2}{R^2}\right)\psi\left(2^j\left(1-\frac{|\xi|^2}{R^2}\right)\right)$.
 
 The boundedness of the square function corresponding to term $\Psi_2\left(\frac{|\xi|^2}{R^2}\right)\left(1-\frac{|\xi|^2}{R^2}-t^2\right)^{\zeta}$ is dealt with in a similar way as in the case of $-1<\gamma<0$ as above. 
 Further, note that $\Psi_1\left(\frac{|\xi|^2}{R^2}\right)$ behaves in a similar manner as $\psi\left(2^j\left(1-\frac{|\xi|^2}{R^2}\right)\right)$.
 Therefore the corresponding inequality \eqref{easypart} can be proved as in the case of $\gamma=0$.
 %square function is also dealt in a similar way as in the case when $-1<\gamma<0$. 
  When $\gamma= m$ or $m+1$ for $m\geq 2$, \eqref{easypart} follows in the exact same manner as in the case $\gamma=1$. We write split the multiplier using binomial expansion of $\left(1-\frac{|\xi|^2}{R^2}-t^2\right)^m$ as in \eqref{gamma} and get the desired estimate \eqref{easypart} as done for $\gamma=1$. When $\gamma\in (m,m+1)$, we write $\gamma=\zeta+m+1, \zeta \in(-1,0)$ and follow the same procedure as in the case when $\gamma\in (0,1)$ above.
 
 \subsection*{Case II: When $t\in [\sqrt{2^{-j-1-\epsilon_0}},\sqrt{2^{-j+1}}]$}
%We will now deal with the $L^p$ boundedness of $$\sup_{R>0}\left(\int_{\sqrt{2^{-j-1-\epsilon}}}^{\sqrt{2^{-j+1}}}|{S}_{j,\gamma+1}^{R,t}f(x)t^{2\delta +1}|^2 dt\right)^{1/2},$$ i.e.  when $t\in [\sqrt{2^{-j-1-\epsilon_0}},
%\sqrt{2^{-j+1}}]$ in the integral over $t$. 
Note that in this case $t\approx 2^{-j/2}.$
We rewrite	
$$\psi\left(2^j\left(1-\frac{|\xi|^2}{R^2}\right)\right)\left(1-\frac{|\xi|^2}{R^2}-t^2\right)_+^{\gamma}=(1-t^2)^{\gamma}\psi\left(2^j\left(1-\frac{|\xi|^2}{R^2}\right)\right)\left(1-\frac{|\xi|^2}{R^2(1-t^2)}\right)_+^{\gamma}$$	
Making a change of variable $1-t^2=s^2$ in  $2^{-j(\delta+1/2)}2^{j/4}\left(\int_{\sqrt{2^{-j-1-\epsilon_0}}}^{\sqrt{2^{-j+1}}}|{S}_{j,\gamma+1}^{R,t}f(x)|^2t dt\right)^{1/2},$ we get the following operator 

$${S'}_{j,\gamma+1}^{R,s}f(x)=\int_{\R^n}\psi\left(2^j\left(1-\frac{|\xi|^2}{R^2}\right)\right)\left(1-\frac{|\xi|^2}{s^2R^2}\right)_+^{\gamma}\hat{f}(\xi)e^{2\pi ix.\xi}d\xi,$$ 
where $t^2\in [2^{-j-1-\epsilon_0}, 2^{-j+1}]$ and $1-t^2=s^2$.

Therefore, we need to establish $L^p$ boundedness of the following square function $$\sup_{R>0}\left(\int_{s_1}^{s_2}|{S'}_{j,\gamma+1}^{R,s}f(x)|^2 sds\right)^{1/2},$$ where $s_1=\sqrt{1-2^{-j+1}}$ and $s_2=\sqrt{1-2^{-j-1-\epsilon_0}}.$ 

Note that we can ignore the term $s$ inside the integral as $s\approx 1$.
Let $M'>100$ be a large number. We shall deal with the cases $j\geq M'$ and $2\leq j<M'$ separately. 
\begin{proposition}\label{maxsquare2} For $j\geq M'$ and $0<\epsilon<1$ we have the following estimate 
	$$\left\|\sup_{R>0}\left(\int^{s_2}_{s_1}|{S'}_{j,\gamma+1}^{R,s}f(\cdot)|^2 ds\right)^{1/2}\right\|_{L^p(\R^n)}\lesssim_{\epsilon} 2^{-j\gamma }2^{j(\alpha(p)-1/2)}2^{\epsilon j}\|f\|_{L^p(\R^n)},$$ 
when $p\geq \mathfrak p_n$ or $p=2$ and $\gamma>\alpha(p)-1/2$. 
\end{proposition}
\begin{proposition}\label{last} When $\gamma>\alpha(p)-1/2$ and $p=2$ or $p\geq \mathfrak p_n$ and $0<\epsilon <1$
$$\left\|\sup_{R>0}\left(\int^{s_2}_{s_1}|{S'}_{j,\gamma+1}^{R,s}f(x)|^2 ds\right)^{1/2}\right\|_{L^p(\R^n)}\leq C(M',\gamma,\epsilon)\|f\|_{L^p(\R^n)}$$
for all $2\leq j\leq M'$.
\end{proposition}
Notice that Theorem \ref{maxsquare} follows from Propositions~\ref{maxsquare2} and \ref{last} with $\gamma=\beta-1$. Therefore, it remains to prove the Propositions~\ref{maxsquare2} and \ref{last}. 
\subsection*{Proof of Proposition~\ref{maxsquare2}}
We perform suitable decomposition of the multiplier  $$\psi\left(2^j\left(1-\frac{|\xi|^2}{R^2}\right)\right)\left(1-\frac{|\xi|^2}{s^2R^2}\right)_+^{\gamma}$$
and reduce our problem to already known forms of the square functions. 

For convenience, we set $R=1$ in the multiplier above and write \begin{eqnarray}\label{seconddec}
\left(1-\frac{|\xi|^2}{s^2}\right)_+^{\gamma}&=&\sum_{k\geq 2}2^{-k\gamma}\tilde{\psi}\left(2^k\left(1-\frac{|\xi|^2}{s^2}\right)\right) +\tilde{\psi}_0\left(\frac{|\xi|^2}{s^2}\right)
\end{eqnarray}
%Fix $0<\epsilon'<\epsilon$. Let us consider the case when $j\leq k\leq (1+\epsilon')j$
where $\tilde{\psi}\in C^{\infty}_0([1/2,2])$ and $\tilde{\psi}_0\in C^{\infty}_0([0,3/4])$. 
Using the above decomposition \eqref{seconddec}
we get
\begin{eqnarray}\label{thirddec}
\psi\left(2^j\left(1-\frac{|\xi|^2}{R^2}\right)\right)\left(1-\frac{|\xi|^2}{s^2}\right)_+^{\gamma}&=\sum_{k\geq 2}2^{-k\gamma}\tilde{\psi}\left(2^k\left(1-\frac{|\xi|^2}{s^2}\right)\right)\psi\left(2^j\left(1-\frac{|\xi|^2}{R^2}\right)\right)\\\nonumber+\tilde{\psi}_0\left(\frac{|\xi|^2}{s^2}\right)\psi\left(2^j\left(1-\frac{|\xi|^2}{R^2}\right)\right)\end{eqnarray}
Let $M>0$ be such that $M^{-1}\tilde{\psi}\in C^N([\frac{1}{2},2])$. Note that $\psi\left(2^j\left(1-|\xi|^2\right)\right)\tilde{\psi}\left(2^k\left(1-\frac{|\xi|^2}{s^2}\right)\right)=0$ whenever $s^2(1-2^{-k-1})<1-2^{-j+1},$ i.e. when $2^{k}<2^{j-1}\frac{s^2}{2-2^{j}(1-s^2)}.$ For $s\in[s_1,s_2]$ we have  $\frac{s^2}{2-2^{-1-\epsilon_0}}\leq \frac{s^2}{2-2^{j}(1-s^2)}$. Therefore, if $2^{k}<2^{j-1}\frac{s^2}{2-2^{-1-\epsilon_0}}$, the product of terms corresponding to $j$ and $k$ in \eqref{thirddec}vanishes. 
This tells us that in the decomposition~\eqref{thirddec} we need to consider only those terms for which $k\geq j-2$. 

Next, we decompose the annular region $\{\xi: s^2(1-2^{-k+1})\leq |\xi|^2\leq s^2(1-2^{-k-1})\}$ which is nothing but the support of $\tilde{\psi}\left(2^k\left(1-\frac{|\xi|^2}{s^2}\right)\right)$, into annuli of smaller width such that the function $\psi\left(2^j\left(1-|\xi|^2\right)\right)$, when restricted to each of these annuli, behaves like a constant.
 
Let $\varphi\in C_c^{\infty}([-1,1])$ be such that $\sum_{m\in\Z}\varphi(x-m)=1$ on $\R.$ Let $L= \sup_{0\leq l\leq N, x\in [-1,1]}\left|\frac{d^l\varphi}{dx^l}\right|.$  For $0<\nu<1$ we can rewrite the sum above as $\sum_{m\in\Z}\varphi(\nu^{-1}(x-\nu m))\equiv 1$ on $\R.$ 

We restrict this identity to the interval $[0,1]$, i.e., we consider 
$\sum_{0\leq k\leq [\nu^{-1}]+1}\varphi(\nu^{-1}(x-\nu m))=1,~ x\in[0,1].$
Using suitable translation and dilation of the function we can get the same identity on the interval  $[s^2(1-2^{-k+1}),s^2(1-2^{-k-1})]$, i.e., we have 
 $$\sum_{0\leq m\leq [\nu^{-1}]+1}\varphi(\nu^{-1}(b-a)^{-1}(x-a-(b-a)\nu m))=1,$$
where $a=s^2(1-2^{-k+1})$ and $b=s^2(1-2^{-k-1}).$  

Let $\nu=\frac{2}{3}2^{-k\epsilon}$ for $0<\epsilon<1$ and note that $(b-a)\nu=2^{-(1+\epsilon)k}s^2.$ 

Putting $x= |\xi|^2$ in the sum above and multiplying by $\tilde{\psi}\left(2^k\left(1-\frac{|\xi|^2}{s^2}\right)\right)$ on both sides, we have the following decomposition $$\tilde{\psi}\left(2^k\left(1-\frac{|\xi|^2}{s^2}\right)\right)=\sum_{0\leq m\leq [\nu^{-1}]+1}\varphi\left(2^{(1+\epsilon)k}\left(\frac{|\xi|^2}{s^2}-1+2^{-k+1}-2^{-(1+\epsilon)k} m\right)\right)\tilde{\psi}\left(2^k\left(1-\frac{|\xi|^2}{s^2}\right)\right).$$
Taking the one dimensional Fourier transform of $\tilde{\psi}$ we have   $$\tilde{\psi}\left(2^k\left(1-\frac{|\xi|^2}{R^2s^2}\right)\right)=\int_{\R}\widehat{\tilde{\psi}}(\mu)e^{2\pi i\left(2^k\left(1-\frac{|\xi|^2}{s^2}\right)\right)\mu} d\mu.$$
We rewrite the exponential term in the integral above in the following way 
$$e^{2\pi i\left(2^k\left(1-\frac{|\xi|^2}{s^2}\right)\mu\right)}=e^{2\pi i2^k(2^{-k+1}-2^{-k(1+\epsilon)}m)\mu}e^{-2\pi i2^{-\epsilon k}2^{(1+\epsilon)k}\left(\frac{|\xi|^2}{s^2}-1+2^{-k+1}-2^{-(1+\epsilon)k} m\right)\mu}.$$
Note that $2^k(2^{-k+1}+2^{-k(1+\epsilon)}m)$ is uniformly bounded for all $m\leq [\nu^{-1}]+1$ and $k\geq 2$.
Denote $$\xi_{k,m}^s=\left(\frac{|\xi|^2}{s^2}-1+2^{-k+1}-2^{-(1+\epsilon)k} m\right).$$ 
Using Taylor's series expansion we get 
$$e^{-2\pi i2^{ k}\xi_{k,m}^s\mu}=\sum_{q=0}^{N-1}\frac{(-1)^q}{q!}(2\pi i2^{-\epsilon k})^q2^{q(1+\epsilon)k}\left(\xi_{k,m}^s\right)^q\mu^q + r_N\left(2^{k}\xi_{k,m}^s\mu\right),$$ 
where $r_N$ is the remainder term. 

For $0\leq q\leq N$ the remainder term $r_N$ satisfies \begin{equation}\label{remainderagain}\left|\frac{d^qr_N}{ds^q}(s)\right|\leq s^{N-q}.\end{equation}
Let us write $2^k(2^{-k+1}+2^{-k(1+\epsilon)}m)=d^{\epsilon}_{k,m}$. Using the expansion and notation as above we get that   
\begin{eqnarray*}
	\tilde{\psi}\left(2^k\left(1-\frac{|\xi|^2}{s^2}\right)\right)&=&\sum_{m\leq [\nu^{-1}]+1}\sum_{q=0}^{N-1}\frac{(-1)^q}{q!}(2\pi i2^{-\epsilon k})^q\frac{d^q\tilde{\psi}}{dx^q}(d^{\epsilon}_{k,m}) \varphi_q(2^{(1+\epsilon)k}\xi^s_{k,m})\\&&
	+\sum_{m\leq [\nu^{-1}]+1}\varphi\left(2^{(1+\epsilon)k}\xi_{k,m}^s\right)\int_{\R}\widehat{\tilde{\psi}}(\mu)e^{2\pi i d^{\epsilon}_{k,m}\mu}~r_N\left(2^{k}\xi_{k,m}^s\eta\right)d\mu,
	\end{eqnarray*}
where $\varphi_q(x)=x^q\varphi(x)$. 

Observe that the condition on $\tilde{\psi}$ and the fact that $d^{\epsilon}_{k,m}$ is uniformly bounded in $k$ and $m$ imply that 
$$\sup_{0\leq q\leq N-1}\left|\frac{d^q\tilde{\psi}}{dx^q}(d^{\epsilon}_{k,m})\right|\leq M,$$ where $M$ is independent of $k$ and $m\leq [\nu^{-1}]+1$.

%Let $\varphi_q^R\left(2^{(1+\epsilon)k}\xi_{k,m}^s\right)=\varphi_q\left(2^{(1+\epsilon)k}\xi_{k,m}^{Rs}\right)$.
Define
\begin{equation}\label{opertorU}U^{\varphi_q,j}_{Rs,R}(k,m)f(x)=\int_{\R^n}\varphi_q\left(2^{(1+\epsilon)k}\xi_{k,m}^{Rs}\right)\psi\left(2^j\left(1-\frac{|\xi|^2}{R^2}\right)\right)\hat{f}(\xi)e^{2\pi ix.\xi} d\xi,\end{equation}
\begin{equation}\label{operatorP}P_{Rs,R}^{j,k}(\mu,m)f(x)=\int_{\R^n}\varphi\left(2^{(1+\epsilon)k}\xi_{k,m}^{Rs}\right)\psi\left(2^j\left(1-\frac{|\xi|^2}{R^2}\right)\right)r_N\left(2^{k}\xi_{k,m}^{Rs}\mu\right)\hat{f}(\xi)e^{2\pi ix.\xi}d\xi.\end{equation}
The decomposition of the symbol corresponding to the operator $S'^{R,s}_{j,\gamma+1}$ (replace $\xi$ by $\xi/R$) as carried out in the discussion above gives us that 
\begin{eqnarray}\label{microdecomposition}
{S'}_{j,\gamma+1}^{R,s}f(x)&=&\sum_{k\geq j-2}2^{-k\gamma}\sum_{m\leq [\nu^{-1}]+1}\sum_{q=0}^{N-1}\frac{(-1)^q}{q!}(2\pi i2^{-\epsilon k})^q\frac{d^q\tilde{\psi}}{dx^q}(d^{\epsilon}_{k,m})U^{\varphi_q,j}_{Rs,R}(k,m)f(x)\\ 
&+&\nonumber \sum_{k\geq j-2}2^{-k\gamma}\sum_{m\leq [\nu^{-1}]+1}\int_{\R}\widehat{\tilde{\psi}}(\mu)e^{2\pi i d^{\epsilon}_{k,m}\mu}P_{Rs,R}^{j,k}(\mu,m)f(x)d\mu.
\end{eqnarray}
Therefore, in order to prove Proposition~\ref{maxsquare2}
we need to deal with the operators $U^{\varphi_q,j}_{Rs,R}(k,m)$ and $P_{Rs,R}^{j,k}(\mu,m)$. We first consider the operator $U^{\varphi_q,j}_{Rs,R}(k,m)$ with $q=0$ and note that the terms with $1\leq q\leq N-1$ can be dealt with similarly as the function $\varphi_q$ behaves in the same way for all $0\leq q\leq N-1$. Also, $\varphi_q$ is a smooth function supported in $[-1,1]$ and it satisfies $\sup_{0\leq l\leq N, x\in [-1,1]}|\frac{d^l\varphi_q}{dx^l}|\leq L q!$.  

In order to prove the required estimates we need to perform another decomposition for the operator $U^{\varphi_0,j}_{Rs,R}(k,m)$. We decompose it into operators whose multipliers are supported in smaller annular regions of width approximately $2^{-(1+\epsilon)k}$. This process leaves a remainder term which can be dealt with easily.  

As previously, we work with $R=1$ and then replace $\xi$ by $\xi/R$. We perform  a similar decomposition to the function $\psi(2^j(1-|\xi|^2))$ into smooth functions having their supports in the annulus $\{\xi: |\xi|^2\in [a+ s^22^{-(1+\epsilon)k}(m-1),a+ s^22^{-(1+\epsilon)k}(m+1)]\}$, where $a=s^2(1-2^{-k+1}).$ 

Consider $\psi(2^j(1-|\xi|^2))=\int_{\R}\hat{\psi}(\mu)e^{2\pi i(2^j(1-|\xi|^2))\mu }d\mu$
and write 
$$e^{2\pi i(2^j(1-|\xi|^2))\mu }=e^{2\pi i2^j(1-a-s^22^{-(1+\epsilon)k}m)\mu}e^{-2\pi i2^js^2(\frac{|\xi|^2}{s^2} -1+2^{-k+1}-2^{-(1+\epsilon)k}m)\mu}.$$
Next we write $2^js^2\xi^s_{k,m}=2^j2^{-(1+\epsilon)k}s^22^{(1+\epsilon)k}\xi^s_{k,m}$ and use Taylor's series expansion of $e^{-2\pi i2^js^2\xi^s_{k,m}\mu}$  to  get that 
\begin{eqnarray*}
	\varphi\left(2^{(1+\epsilon)k}\xi_{k,m}^s\right)\psi(2^j(1-|\xi|^2))&=&\sum_{l\geq 0}^{N-1}\frac{1}{l!}s^22^{-d(1+\epsilon)l}2^{-j\epsilon l}\varphi_l\left(2^{(1+\epsilon)k}\xi_{k,m}^s\right)\frac{d^l\psi}{dx^l}\left(2^j(1-a-s^22^{-(1+\epsilon)k}m)\right)\\&&+\varphi\left(2^{(1+\epsilon)k}\xi_{k,m}^s\right) \int_{\R}\hat{\psi}(\mu)e^{2\pi i\left(2^j(1-a-s^22^{-(1+\epsilon)k}m)\mu\right)}r_N\left(2^js^2\xi_{k,m}^s\mu\right)d\mu.
	\end{eqnarray*}
%where $\varphi_l(x)=x^l\varphi(x)$ and $\psi^l(x)=\frac{d^l\psi}{dx^l}(x)$. 
Recall that here $k=j+d$ and $\varphi_l(x)=x^l \varphi(x)$.

Consider the following operators  $$V^{\varphi_l}_{Rs}(k,m)f(x)=\int_{\R^n}\varphi_l\left(2^{(1+\epsilon)k}\xi_{k,m}^{Rs}\right)\hat{f}(\xi)e^{2\pi ix.\xi} d\xi,$$
$$Q_{N,R,s}^{k}(\mu)f(x)=\int_{\R^n}\varphi\left(2^{(1+\epsilon)k}\xi_{k,m}^{Rs}\right)r_N\left(2^js^2\xi_{k,m}^{Rs}\mu\right)\hat{f}(\xi)e^{2\pi ix.\xi}d\xi.$$
Finally, replace $\xi$ by $\xi/R$ in the decomposition of $U^{\varphi_0,j}_{Rs,R}(k,m)$ as above to get that \begin{equation}\label{decompositionofU}U^{\varphi_0,j}_{Rs,R}(k,m)f(x)=\sum_{l\geq 0}^{N-1}\frac{1}{l!}s^22^{-d(1+\epsilon)l}2^{-j\epsilon l}\psi^l(2^j(1-a-s^22^{-(1+\epsilon)k}m))V^{\varphi_l}_{Rs}(k,m)f(x) + X_{N,R,s}^kf(x),
\end{equation}
where $$X_{N,R,s}^kf(x)=\int_{\R}\hat{\psi}(\mu)e^{2\pi i2^j(1-a-s^22^{-(1+\epsilon)k}m)\mu}Q_{N,R,s}^{k}(\mu)f(x) d\mu.$$
It is easy to see that 
$$\sup_{0\leq l\leq N-1}\left|\frac{d^l\psi}{dx^l}(2^j(1-a-s^22^{-(1+\epsilon)k}m))\right|\leq D,$$ where $D$ is independent of $j,k$,  $s\in[s_1,s_2]$ and $m\leq [\nu^{-1}]+1$.

Now we have the following estimate for the operator $V^{\varphi_l}_{Rs}(k,m)$. 
\begin{lemma}\label{reduction1} Let $p$ and $\alpha(p)$ be as in Proposition~\ref{maxsquare2}. Then the following estimate holds. 
	$$\left\|\sup_{R>0}\left(\int_{s_1}^{s_2}|V^{\varphi_l}_{Rs}(k,m)f(\cdot)|^2ds\right)^{1/2}\right
	\|_{L^p(\R^n)}\leq C(p)l!2^{(1+\epsilon)(\alpha(p)-1/2)k}\|f\|_{L^p(\R^n)}.$$
\end{lemma}
\begin{proof} Recall that $$\varphi_l\left(2^{(1+\epsilon)k}\xi_{k,m}^{Rs}\right)=\varphi_l\left(2^{(1+\epsilon)k}\left(\frac{|\xi|^2}{R^2s^2}-1+2^{-k+1}-2^{-(1+\epsilon)k} m\right)\right).$$
Denote $(1-2^{-k+1}-2^{-(1+\epsilon)k}m)=c_{k,m}^{\epsilon}$. Note that for fixed $\epsilon>0$ and $k\geq 10$, $c_{k,m}^{\epsilon}$ has a uniform lower and upper bound in $k$ and for all $m\leq [\nu^{-1}]+1$.
Since  $2^{(1+\epsilon)k}\xi^s_{k,m}=-2^{(1+\epsilon)k}c_{k,m}^{\epsilon}\left(1-\frac{|\xi|^2}{s^2c_{k,m}^{\epsilon}}\right),$	
We can rewrite $$\varphi_l\left(2^{(1+\epsilon)k}\xi_{k,m}^{Rs}\right)=\varphi_l\left(-\varrho^{-1}\left(1-\frac{|\xi|^2}{R^2s^2c_{k,m}^{\epsilon}}\right)\right),$$
where $\varrho^{-1}=2^{(1+\epsilon)k}c_{k,m}^{\epsilon}$.

After making a change of variable $s\rightarrow R^{-1}(c_{k,m}^{\epsilon})^{-1}s$ in  $\left(\int_{s_1}^{s_2}|V^{\varphi_l}_{Rs}(k,m)f(\cdot)|^2ds\right)^{1/2}$ observe that it can be dominated by $$\sup_{R>0}\left(R^{-1}\int_{Rc^{\epsilon}_{k,m}s_1}^{Rc^{\epsilon}_{k,m}s_2}|B_{\varrho,Rc^{\epsilon}_{k,m}s}^{\varphi_l}f(x)|^2ds\right)^{1/2}.$$
Further, the quantity above can be dominated by the following square function  $$\left(\int_{0}^{\infty}|B_{\varrho,s}^{\varphi_l,}f(x)|^2s^{-1}ds\right)^{1/2}.$$
The boundedness of the square function as above can be deduced from \cite{Leesqr} with the desired bound, see Lemma \ref{localisedversion} for  detail. 
\end{proof}
Next, we have the following estimate for the operator $X^k_{N,Rs}$.
\begin{lemma}\label{reduction2} For $1<p\leq \infty$, the following holds. 
	$$\left\|\sup_{R>0,~s\in[s_1,s_2]}|X^k_{N,R,s}f|\right\|_{L^p(\R^n)}\lesssim 2^{(1+\epsilon)(\alpha(p)-1/2)k}\|f\|_{L^p(\R^n)}.$$
\end{lemma}

\begin{proof} Recall the definition of $X^k_{N,R,s}$ and observe that it is enough to prove that 
%We know that \begin{equation}\label{rel}
% X_{N,Rs}^kf(x)=\int_{\R}\hat{\psi}(\mu)e^{2\pi i2^j(1-a-s^22^{-(1+\epsilon)k}m)\mu}Q_{N,Rs}^{k}(\mu)f(x) d\mu,\end{equation}
%where $Q_{N,Rs}^{k}(\mu)f(x)=\int_{\R^n}\varphi\left(2^{(1+\epsilon)k}\xi_{k,m}^{Rs}\right)r_N\left(2^js^2\xi_{k,m}^{Rs}\mu\right)\hat{f}(\xi)e^{2\pi ix.\xi}d\xi$.
%
%\vspace{0.1in}
%We will first show that 	
\begin{equation}\label{maxinequality}\left\|\sup_{R>0,~s\in[s_1,s_2]}|Q_{N,R,s}^{k}(\mu)f|\right\|_{L^p(\R^n)}\lesssim (1+|\mu|)^{N}2^{-j\epsilon N}2^{-d(1+\epsilon)N}2^{(1+\epsilon)kn}\|f\|_{L^p(\R^n)}
\end{equation} for all $1<p\leq \infty$.
%The boundedness of $\sup_{R>0,~s\in[s_1,s_2]}|X_{N,Rs}^kf|$ follows immediately by combining \eqref{maxinequality} and \eqref{rel} as $\hat{\psi}$ is a Schwartz class function. 

Let $k=j+d$ and 
define 
\begin{equation}\label{multiplierform}
M_{s,\mu}(\xi)=2^{j\epsilon N}2^{d(1+\epsilon)N}\varphi\left(2^{(1+\epsilon)k}\xi_{k,m}^s\right)r_N\left(2^js^2\xi_{k,m}^s\mu\right).
\end{equation}
We also define $M_{R,s,\mu}(\xi)=M_{s,\mu}(\xi/R)$. 
Note that $Q_{N,R,s}^{k}(\mu)f=2^{-j\epsilon N}2^{-d(1+\epsilon)N} f* (\mathcal{F}^{-1}M_{R,s,\mu})$. Here $\mathcal{F}^{-1} f$ denotes the inverse Fourier transform of the function $f$. 

From the estimates on $r_N$ in \eqref{remainderagain}, it is easy to see that $|\frac{\partial^{\beta}}{\partial\xi^{\beta}}M_{s,\mu}(\xi)|\leq c L (1+|\mu|)^N 2^{(1+\epsilon)k|\beta|}$ for all $|\beta|\leq N.$ 
In particular, when $\beta=0$,  we have 
$$|M_{s,\mu}(\xi)|\leq 2^{j\epsilon N}2^{d(1+\epsilon)N} |\varphi\left(2^{(1+\epsilon)k}\xi_{k,m}^s\right)|~|2^js^2\xi_{k,m}^s\mu|^N.$$
We know that $\varphi\left(2^{(1+\epsilon)k}\xi_{k,m}^s\right)$ is non-zero for $|\xi_{k,m}^s|\leq 2^{-(1+\epsilon)k}$ and $\|\varphi\|_{\infty}\leq L$. This gives us the desired estimate for $\frac{\partial^{\beta}}{\partial\xi^{\beta}}M_{s,\mu}(\xi)$ when $\beta=0$.

Note that the function $M_{s,\mu}(\xi)$ is supported in a cube of sidelength at most $2$ for all $s\in [s_1,s_2]$, $m\leq [\nu^{-1}]+1$ and $\mu.$

An integration by parts argument implies that \begin{equation}\label{pointwiseestimate}
	|\mathcal{F}^{-1}M_{s,\mu}(x)|\leq c L (1+|\mu|)^N (1+2^{-(1+\epsilon)k}|x|)^{-n-1}
\end{equation}
holds uniformly in $s$.
%Recall from \eqref{multiplierform} that $$Q_{N,Rs}^{k}(\mu)f(x)=2^{-j\epsilon N}2^{-d(1+\epsilon)N}\int_{\R^n}m_{s,\mu}(R^{-1}\xi)\hat{f}(\xi)e^{2\pi ix.\xi}d\xi.$$
The estimate above yields the pointwise estimate 
$$\sup_{R>0,~ s\in[s_1,s_2]}|Q_{N,R,s}^{k}(\mu)f(x)| \lesssim (1+|\mu|)^N2^{-j\epsilon N}2^{-d(1+\epsilon)N}2^{k(1+\epsilon)n}\mathcal{M}f(x).$$
Choosing $N$ large enough we get the desired estimate. 
\end{proof}
Combining the results from Lemmata \ref{reduction1} and \ref{reduction2} and putting them in \eqref{decompositionofU} we get that 
\begin{equation}\label{boundednessofU}
\left\|\sup_{R>0}\left(\int_{s_1}^{s_2}|U^{\varphi_q,j}_{Rs,R}(k,m)f(\cdot)|^2ds\right)^{1/2}\right\|_p\lesssim 2^{(1+\epsilon)(\alpha(p)-1/2)k}\|f\|_p	
\end{equation}
for the same range of $p$ as in Lemma \ref{reduction1}.

Now we deal with the remaining part in \eqref{microdecomposition} which is given by 
$$\sum_{m\leq [\nu^{-1}]+1}\int_{\R}\widehat{\tilde{\psi}}(\mu)e^{2\pi i d^{\epsilon}_{k,m}\mu}P_{Rs,R}^{j,k}(\mu,m)f(x)d\mu.$$ 
%Enough to look at $$\int_{\R}\hat{\psi}(\mu)e^{2\pi i d^{\epsilon}_{k,m}\mu}P_{Rs,R}^{j,k}(\mu,m)f(x)d\mu$$
%for $m\in \Z\cap[0,\delta^{-1}+1]$.

Recall from~\eqref{operatorP} that $$P_{Rs,R}^{j,k}(\mu,m)f(x)=\int_{\R^n}\varphi\left(2^{(1+\epsilon)k}\xi_{k,m}^{Rs}\right)\psi\left(2^j\left(1-\frac{|\xi|^2}{R^2}\right)\right)r_N\left(2^{k}\xi_{k,m}^{Rs}\mu\right)\hat{f}(\xi)e^{2\pi ix.\xi}d\xi.$$
Fix $m\leq [\nu^{-1}]+1$ 
and decompose $\psi\left(2^j\left(1-\frac{|\xi|^2}{R^2}\right)\right)$ into a sum of smooth functions supported in an annulus of width of the order $2^{-(1+\epsilon)k}$ and a remainder term. This is similar to the previous case. We get that
\begin{align}\label{decompositionofP}
P_{Rs,R}^{j,k}(\mu,m)f(x)=\sum_{l=0}^{N-1}\frac{1}{l!}s^22^{-d(1+\epsilon)l}2^{-j\epsilon l}\psi_j(k,m,l)D_{s,R}^{l,k}(\mu,m)f(x)\\\nonumber+	\int_{\R}\int_{\R^n}\hat{\psi}(\mu')s_N^R(\mu,\mu')\varphi(2^{(1+\epsilon)k}\xi_{k,m}^{Rs})\hat{f}(\xi)e^{2\pi ix.\xi} e^{2\pi i2^j(1-a-s^22^{-(1+\epsilon)k}m)\mu'}d\mu'd\xi,
\end{align}
where $\psi_j(k,m,l)=\frac{d^l\psi}{dx^l}(2^j(1-a-s^22^{-(1+\epsilon)k}m)),~ s_N^R(\mu,\mu')=r_N(2^k\xi_{k,m}^{Rs}\mu)r_N(2^js^2\xi_{k,m}^{Rs}\mu')$ and 
 \begin{equation}\label{microterms1}D_{s,R}^{l,k}(\mu,m)f(x)=\int_{\R^n}\varphi_l(2^{(1+\epsilon)k}\xi_{k,m}^{Rs})r_N(2^k\xi_{k,m}^{Rs}\mu)e^{2\pi ix.\xi}\hat{f}(\xi)d\xi\end{equation} for $0\leq l\leq N-1$ 
% and \begin{equation}\label{microterms2}\varphi(2^{(1+\epsilon)k}\xi_{k,m}^{Rs}) \int_{\R}\hat{\psi}(\mu')s_N^R(\mu,\mu')e^{2\pi i2^j(1-a-s^22^{-(1+\epsilon)k}m)\mu'}d\mu',\end{equation}

It is not difficult to see that the maximal function (with respect to $R$ ) associated with $\int_{\R}\hat{\psi}(\mu)e^{2\pi i d^{\epsilon}_{k,m}\mu}D_{s,R}^{l,k}(\mu,m)f(x)d\mu$ can be dealt with in a similar manner as we did for the maximal function $\sup_{R>0~ s\in[s_1,s_2]}|X_{N,R,s}^kf|$ in Lemma \ref{reduction2}. Moreover, the bound is independent of $k,m$ and $s$ and grows at most like $l!$ in $l$.

%Hence we also get the boundedness of $$
Define $E(\mu,\mu')=e^{2\pi i d^{\epsilon}_{k,m}\mu}e^{2\pi i2^j(1-a-s^22^{-(1+\epsilon)k}m)\mu'}$ and consider  \begin{equation}\label{remainder}I_{N,Rs}^kf(x)=\int_{\R^2}\widehat{\tilde{\psi}}(\mu)\hat{\psi}(\mu')E(\mu,\mu')T_{N}^{Rs}(\mu,\mu')f(x)d\mu d\mu',\end{equation}
where $$T_{N}^{Rs}(\mu,\mu')f(x)=\int_{\R^n}s_N^R(\mu,\mu')\varphi(2^{(1+\epsilon)k}\xi_{k,m}^{Rs})e^{2\pi ix.\xi}\hat{f}(\xi)d\xi.$$
For $1<p\leq \infty$, we have the following estimate for $I_{N,Rs}^kf$

\begin{eqnarray}\label{reduction3} \left\|\sup_{R>0,~s\in[s_1,s_2]}|I_{N,Rs}^kf|\right\|_{L^p(\R^n)}\lesssim 2^{-j\epsilon N}2^{-d(1+\epsilon)N}2^{-k(1+\epsilon)N}2^{(1+\epsilon)kn}\|f\|_{L^p(\R^n)}
\end{eqnarray}
Now from \eqref{remainder} we know that it is enough to show that
$$\left\|\sup_{R>0,~s\in[s_1,s_2]}|T_{N}^{Rs}(\mu,\mu')f|\right\|_{L^p(\R^n)}\leq (1+|\mu_1|+|\mu_2|)^{N}2^{-j\epsilon N}2^{-d(1+\epsilon)N}2^{-k(1+\epsilon)N}2^{(1+\epsilon)kn}\|f\|_{L^p(\R^n)}.$$ Define $$M_{\mu,\mu'}^{R,s}(\xi)=2^{j\epsilon N}2^{d(1+\epsilon)N}2^{k(1+\epsilon)N}s_N^R(\mu,\mu')\varphi(2^{(1+\epsilon)k}\xi_{k,m}^{Rs}) .$$
It can be verified that $M_{\mu,\mu'}^{R,s}(\xi)=M_{\mu,\mu'}^{1,s}(\xi/R)$.
Further, the inverse Fourier transform of $M_{\mu,\mu'}^{1,s}(\xi)$ satisfies estimate similar to the one in \eqref{pointwiseestimate}.
Therefore, the boundedness of the corresponding maximal function follows in a similar fashion as in Lemma \ref{reduction2}.

 By the estimate \eqref{reduction3} and the explanation provided as earlier we get the desired estimate for the operator $$\int_{\R}\widehat{\tilde{\psi}}(\mu)e^{2\pi i d^{\epsilon}_{k,m}\mu}P_{Rs,R}^{j,k}(\mu,m)f(x)d\mu.$$
 Combining all the results together in \eqref{microdecomposition} we complete the proof of Proposition~\ref{maxsquare2}.
 \qed

\subsection*{Proof of Proposition~\ref{last}}
Note that \begin{eqnarray*}
	{S'}_{j,\gamma+1}^{R,s}f(x)&=&\int_{\R^n}\psi\left(2^j\left(1-\frac{|\xi|^2}{R^2}\right)\right)\left(1-\frac{|\xi|^2}{R^2s^2}\right)_+^{\gamma}\hat{f}(\xi)e^{2\pi ix.\xi}d\xi\\
	&=&B^{\psi}_{2^{-j},R}B_{Rs}^{\gamma}f(x),
	\end{eqnarray*}
where  $$B^{\psi}_{2^{-j},R}(f)(x)=\int_{\R^n}\psi\left(2^j\left(1-\frac{|\xi|^2}{R^2}\right)\right)\hat{f}(\xi)e^{2\pi ix.\xi} d\xi.$$
As done in Lemma \ref{steinlemma} we have that 
\begin{align*}\left(\sup_{R>0}\int_{s_1}^{s_2}|B^{\psi}_{2^{-j},R} B_{Rs}^{\gamma}f(x)|^2 ds\right)^{1/2}\leq \sum_{k=1}^{n_0}\left(\sup_{R>0}\int_{s_1}^{s_2}|B^{\psi}_{2^{-j},R} \left(B_{Rs}^{\gamma+k}-B_{Rs}^{\gamma+k-1}\right)f(x)|^2 ds\right)^{1/2}\\ +\sup_{R>0}\left(\int_{s_1}^{s_2}|B^{\psi}_{2^{-j},R} B_{Rs}^{\gamma+n_0}f(x)|^2 ds\right)^{1/2},	
 \end{align*}
where $n_0$ is such that $\gamma+n_0>\frac{n-1}{2}.$ 

For a fixed $1\leq k\leq n_0$ let us first consider  $$\left(\sup_{R>0}\int_{s_1}^{s_2}|B^{\psi}_{2^{-j},R} \left(B_{Rs}^{\gamma+k}-B_{Rs}^{\gamma+k-1}\right)f(x)|^2 ds\right)^{1/2}$$ 
 After applying the change of variable $s\rightarrow Rs$ the quantity above can be dominated by the following square function  $$\left(\sup_{R>0}\int_{0}^{\infty}|B^{\psi}_{2^{-j},R} \left(B_{s}^{\gamma+k}-B_{s}^{\gamma+k-1}\right)f(x)|^2 s^{-1}ds\right)^{1/2}.$$
 Since $\psi\left(2^j\left(1-\frac{|\xi|^2}{R^2}\right)\right)$ is a compactly supported smooth function. An integration by parts argument tells us that the $L^1$ norm of it's inverse Fourier transform is bounded by $2^{jn}$. Therefore, we get the pointwise estimate. 
 $\sup_{R>0}|B^{\psi}_{2^{-j},R}f(x)|\leq 2^{jn}\mathcal{M}f(x)$ and consequently we have that 
 $$\left(\sup_{R>0}\int_{0}^{\infty}|B^{\psi}_{2^{-j},R} \left(B_{s}^{\gamma+k}-B_{s}^{\gamma+k-1}\right)f(x)|^2 s^{-1}ds\right)^{1/2}\leq 2^{jn}\left(\int_{0}^{\infty}|\mathcal{M} \left(B_{s}^{\gamma+k}-B_{s}^{\gamma+k-1}\right)f(x)|^2 s^{-1}ds\right)^{1/2}.$$
 Invoking the vector valued boundedness of the Hardy-Littlewood Maximal function, see \cite{DK} for detail, the right hand side quantity in the inequality above is bounded on $L^p(\R^n,\int_0^{\infty}f_s(\cdot)s^{-1}ds)$ for all $1<p\leq \infty.$ This yields  \begin{eqnarray*}&& \left\|\left(\sup_{R>0}\int_{0}^{\infty}|B^{\psi}_{2^{-j},R}\left(B_{s}^{\gamma+k}-B_{s}^{\gamma+k-1}\right)f(x)|^2 s^{-1}ds\right)^{1/2}\right\|_p \\&\leq &2^{jn}\left\|\left(\int_{0}^{\infty}| \left(B_{s}^{\gamma+k}-B_{s}^{\gamma+k-1}\right)f(x)|^2 s^{-1}ds\right)^{1/2}\right\|_p.
 	\end{eqnarray*}
 	We can also deduce the above vector valued inequality for the Hardy Lttlewood Maximal function from the result of Feffereman and Stein \cite {FS} as here we are dealing with a separable Hilbert space valued functions and any separable Hilbert space is isomorphic to $l^2(\N)$. 
 The quantity on the right side in the inequality above is nothing but the square function for the Bochner-Riesz means, which it is bounded on $L^p(\R^n)$ for $\gamma>\alpha(p)-1/2$ and $p\geq \mathfrak p_n$ or $p=2.$
 
 Next, we consider the maximal function
 $$\sup_{R>0}\left(\int_{s_1}^{s_2}|B^{\psi}_{2^{-j},R} B_{Rs}^{\gamma+n_0}f(x)|^2 ds\right)^{1/2}.$$
 Again a change of variable $s\rightarrow sR$ in the above gives $$\sup_{R>0}\left(R^{-1}\int_{Rs_1}^{Rs_2}|B^{\psi}_{2^{-j},R} B_{s}^{\gamma+n_0}f(x)|^2 ds\right)^{1/2}.$$
 Since $B^{\psi}_{2^j,R}f(x)\leq 2^{jn}\mathcal{M}f(x)$ a.e. we get that  $$|B_s^{\gamma+n_0}B^{\psi}_{2^{-j},R}f(x)|\leq 2^{jn} \int_{\R^n}|K_s(x-y)|\mathcal{M}f(y)dy,$$ where $K_s$ is the kernel associated with the operator $B_s^{\gamma+n_0}$ and hence the $L^p$ boundedness of $\sup_{s>0,R>0}|B_s^{\gamma+n_0}B^{\psi}_{2^{-j},R}f|$ follows for $1<p\leq \infty$ as $\gamma+n_0$ is large enough to guarantee that $B_s^{\gamma+n_0}$ is dominated by the Hardy-Littlewood maximal function in a pointwise a.e. sense. 
\qed

\subsection{Boundedness of the maximal function  $\mathcal{B}^{\alpha}_{0,*}(f,g)$}\label{secondmaximal}
In this subsection we prove the boundedness of the maximal function  $\mathcal{B}^{\alpha}_{0,*}(f,g).$ In particular, we show that 
for $n\geq 2$ and $\alpha>\alpha_*(p_1,p_2)$, the following holds \begin{eqnarray}\label{secondmax}
	\|\mathcal{B}^{\alpha}_{0,*}(f,g)|\|_{L^p(\R^n)}\lesssim \|f\|_{L^{p_1}(\R^n)}\|g\|_{L^{p_2}(\R^n)},
\end{eqnarray}
where $p_i=2$ or $p_i\geq \mathfrak p_n$ for $i=1,2$ and $\frac{1}{p}=\frac{1}{p_1}+\frac{1}{p_2}$.

We decompose the multiplier $m_{0,R}(\xi,\eta)$ using the partition of identity in the $\eta$ variable. This is same as we did earlier in the $\xi$ variable.
We get that \begin{align*} m_{0,R}(\xi,\eta)=\sum_{j\geq 2}\psi_0\left(\frac{|\xi|^2}{R^2}\right)\psi\left(2^j\left(1-\frac{|\eta|^2}{R^2}\right)\right)\left(1-\frac{|\eta|^2}{R^2}\right)_+^{\alpha}\left(1-\frac{|\xi|^2}{R^2}\left(1-\frac{|\eta|^2}{R^2}\right)^{-1}\right)^{\alpha}_+\nonumber \\+\psi_0\left(\frac{|\xi|^2}{R^2}\right)\psi_0\left(\frac{|\eta|^2}{R^2}\right)\left(1-\frac{|\xi|^2+|\eta|^2}{R^2}\right)_+^{\alpha}.\end{align*}
%Note that the support of the function $\psi_0$ is contained in the set $[0,3/4]$. As a consequence $\psi_0\left(\frac{|\xi|^2}{R^2}\right)\psi_0\left(\frac{|\eta|^2}{R^2}\right)\left(1-\frac{|\xi|^2+|\eta|^2}{R^2}\right)_+^{\alpha}$ is not a smooth function.
We further split $\psi_0(x)=\psi_0^1(x)+\psi_0^2(x)$ such that support of $\psi_0^1$ is contained in the interval $[0,3/16]$ and support of $\psi_0^2$ is contained in the interval $[\frac{3}{32},\frac{3}{4}]$. This gives us the follwoing decomposition of $m_{0,R}$ 
\begin{align}\label{flippeddecomposition} m_{0,R}(\xi,\eta)=\sum_{j\geq 2}\psi_0\left(\frac{|\xi|^2}{R^2}\right)\psi\left(2^j\left(1-\frac{|\eta|^2}{R^2}\right)\right)\left(1-\frac{|\eta|^2}{R^2}\right)_+^{\alpha}\left(1-\frac{|\xi|^2}{R^2}\left(1-\frac{|\eta|^2}{R^2}\right)^{-1}\right)^{\alpha}_+\nonumber \\+\psi_0\left(\frac{|\xi|^2}{R^2}\right)\psi_0^1\left(\frac{|\eta|^2}{R^2}\right)\left(1-\frac{|\xi|^2+|\eta|^2}{R^2}\right)_+^{\alpha}+\psi_0\left(\frac{|\xi|^2}{R^2}\right)\psi_0^2\left(\frac{|\eta|^2}{R^2}\right)\left(1-\frac{|\xi|^2+|\eta|^2}{R^2}\right)_+^{\alpha}.\end{align}

Note that $\psi_0\left(\frac{|\xi|^2}{R^2}\right)\psi_0^1\left(\frac{|\eta|^2}{R^2}\right)\left(1-\frac{|\xi|^2+|\eta|^2}{R^2}\right)_+^{\alpha}$ is a smooth function supported in a ball of radius $\sqrt{\frac{15}{16}}R$. Therefore, it can be easily proved that the maximal function  \begin{equation}\label{trivial}\sup_{R>0}\left|\int_{\R^{n}\times\R^n}\psi_0\left(\frac{|\xi|^2}{R^2}\right)\psi_0^1\left(\frac{|\eta|^2}{R^2}\right)\left(1-\frac{|\xi|^2+|\eta|^2}{R^2}\right)_+^{\alpha}\hat{f}(\xi)\hat{g}(\eta)e^{2\pi ix.(\xi+\eta)}d\xi d\eta\right| \end{equation} 
is bounded from $L^{p_1}(\R^n)\times L^{p_2}(\R^n)\rightarrow L^p(\R^n)$ for all $1<p_i\leq \infty, i=1,2$.

Next, we deal with the maximal functions corresponding to the remaining terms in the decomposition \eqref{flippeddecomposition}.

Let $K$ be a large number such that for all $j\geq K$, the support of $\left(1-\frac{|\xi|^2}{R^2}\left(1-\frac{|\eta|^2}{R^2}\right)^{-1}\right)^{\alpha}_+$ is contained in $|\xi|\leq \frac{R}{8}$ where $\eta$ belongs to the support of the function $\psi\left(2^j\left(1-\frac{|\eta|^2}{R^2}\right)\right).$
We know that $\psi_0\left(\frac{|\xi|^2}{R^2}\right)\equiv 1$ when $|\xi|\leq \frac{R}{8}.$
Consider the operator  $$\tilde{T}^{\alpha}_{R,K}(f,g)(x)=\int_{\R^n\times\R^n}\tilde{m}_{R,K}^{\alpha}(\xi,\eta)\hat{f}(\xi)\hat{g}(\eta)e^{2\pi ix.(\xi+\eta)}d\xi d\eta,$$
where $$\tilde{m}_{R,K}^{\alpha}(\xi,\eta)=\sum_{j\geq K}\psi_0\left(\frac{|\xi|^2}{R^2}\right)\psi\left(2^j\left(1-\frac{|\eta|^2}{R^2}\right)\right)\left(1-\frac{|\eta|^2}{R^2}\right)_+^{\alpha}\left(1-\frac{|\xi|^2}{R^2}\left(1-\frac{|\eta|^2}{R^2}\right)^{-1}\right)^{\alpha}_+.$$
The boundedness of the maximal function $\sup_{R>0}|\tilde{T}^{\alpha}_{R,K}(f,g)|$ can be obtained in a similar fashion as in the previous section for the maximal function $\sup_{R>0}|\sum_{j\geq M'}T^{\alpha}_{j,R}(f,g)|.$ Therefore, we get that  $\sup_{R>0}|\tilde{T}^{\alpha}_{R,K}(f,g)|$ is bounded from $L^{p_1}(\R^n)\times L^{p_2}(\R^n)\rightarrow L^{p}(\R^n)$ when  $\alpha>\alpha_*(p_1,p_2)$ and $p_i=2$ or $p_i\geq \mathfrak p_n$ for $i=1,2$.

Now it remains to prove the desired estimate for the terms corresponding to $2\leq j<K$. We consider  \begin{align*}m'_{R}(\xi,\eta)=\sum_{j\geq 2}^{K}\psi_0\left(\frac{|\xi|^2}{R^2}\right)\psi\left(2^j\left(1-\frac{|\eta|^2}{R^2}\right)\right)\left(1-\frac{|\xi|^2+|\eta|^2}{R^2}\right)_+^{\alpha}\\ +\psi_0\left(\frac{|\xi|^2}{R^2}\right)\psi_0^2\left(\frac{|\eta|^2}{R^2}\right)\left(1-\frac{|\xi|^2+|\eta|^2}{R^2}\right)_+^{\alpha}.\end{align*}
Let $\tilde{T}^{\alpha}_R(f,g)$ be the bilinear operator corresponding to the multiplier $m'_{R}$. i.e.,  $$\tilde{T}^{\alpha}_R(f,g)(x)=\int_{\R^n}\int_{\R^n}m'_{R}(\xi,\eta)\hat{f}(\xi)\hat{g}(\eta)e^{2\pi ix.(\xi+\eta)}d\xi d\eta.$$
Note that for each fixed $2\leq j<K$ it is enough to deal with the operator associated with the multipiler  $$m\left(\frac{\xi}{R},\frac{\eta}{R}\right)=\psi_0\left(\frac{|\xi|^2}{R^2}\right)\psi\left(2^j\left(1-\frac{|\eta|^2}{R^2}\right)\right)\left(1-\frac{|\eta|^2}{R^2}\right)_+^{\alpha}\left(1-\frac{|\xi|^2}{R^2}\left(1-\frac{|\eta|^2}{R^2}\right)^{-1}\right)^{\alpha}_+.$$
This is precisely the situation in section~\ref{sec:dec} when we decompose the bilinear Bochner-Riesz multiplier. We perform the same decomposition to the multiplier $\left(1-\frac{|\xi|^2}{R^2}\left(1-\frac{|\eta|^2}{R^2}\right)^{-1}\right)^{\alpha}_+$ as we did in \eqref{thankstostein} with the roles of $\xi$ and $\eta$ interchanged. This gives us an expression similar to \eqref{decomope} with the roles of $f$ and $g$ interchanged and an extra multiplier term  $\psi_0\left(\frac{|\xi|^2}{R^2}\right)$ with $\hat f$. Define \begin{equation}\label{reversed}B^{\psi_0}_{1,R}f(x)=\int_{\R^n}\psi_0\left(\frac{|\xi|^2}{R^2}\right)\hat{f}(\xi)e^{2\pi ix.\xi} d\xi.
\end{equation}
Observe that it is enough to prove that the maximal function  $\left(\sup_{R>0}R_j^{-1}\int_0^{R_j}|B^{\psi_0}_{1,R} B_t^{\delta}f(x)|^2 dt\right)^{1/2}$ is bounded on $L^{p_1}(\R^n)$ for $\delta>\alpha(p_1)-1/2$ and $p_1=2$ or $p_1\geq \mathfrak p_n. $ The term corresponding to the function $g$ can be dealt with by using the similar method as we did to prove estimate \eqref{easypart} and Proposition \ref{last}. Again following the same technique as in Lemma \ref{steinlemma}, we get the following estimate 
\begin{eqnarray*}\left(\sup_{R>0}R_j^{-1}\int_0^{R_j}|B^{\psi_0}_{1,R} B_t^{\delta}f(x)|^2 dt\right)^{1/2}&\leq& \sum_{k=1}^{n_0}\left(\sup_{R>0}R_j^{-1}\int_0^{R_j}|B^{\psi_0}_{1,R}\left(B_t^{\delta+k}-B_t^{\delta+k-1}\right)f(x)|^2 dt\right)^{1/2}\\&&+ \sup_{R>0}\left(R_j^{-1}\int_{0}^{R_j}|B^{\psi_0}_{1,R}B_t^{\delta+n_0}f(x)|^2 dt\right)^{1/2},
\end{eqnarray*}
where $n_0$ is such that $\delta+n_0>\frac{n-1}{2}$. 

Moreover, we proceed as in Proposition \ref{last}  to get that the maximal function $$\left(\sup_{R>0}R_j^{-1}\int_0^{R_j}|B^{\psi_0}_{1,R} B_t^{\delta}f(x)|^2 dt\right)^{1/2}$$ is bounded on $L^{p_1}(\R^n)$ for $\delta>\alpha(p_1)-1/2$ and $p_1=2$ or $p_1\geq \mathfrak p_n$. We skip the details here to avoid a repetition. 

Let us now consider the term corresponding to the multiplier  $$\widetilde{M}\left(\frac{\xi}{R},\frac{\eta}{R}\right)=\psi_0\left(\frac{|\xi|^2}{R^2}\right)\psi_0^2\left(\frac{|\eta|^2}{R^2}\right)\left(1-\frac{|\eta|^2}{R^2}\right)_+^{\alpha}\left(1-\frac{|\xi|^2}{R^2}\left(1-\frac{|\eta|^2}{R^2}\right)^{-1}\right)^{\alpha}_+$$
Indeed, using \eqref{thankstostein} with the roles of $\xi$ and $\eta$ interchanged, we get that $$\widetilde{M}\left(\frac{\xi}{R},\frac{\eta}{R}\right)= \psi_0\left(\frac{|\xi|^2}{R^2}\right)\psi_0^2\left(\frac{|\eta|^2}{R^2}\right)R^{-2\alpha}\int_0^{uR}\left(R^2\varphi_R(\eta)-t^2\right)_+^{\beta-1}t^{2\delta+1}\left(1-\frac{|\xi|^2}{t^2}\right)^{\delta}_+dt,$$
where $\varphi_R(\eta)=\left(1-\frac{|\eta|^2}{R^2}\right)_+$ and $u=\sqrt{\frac{29}{32}}$. 

This part is handled using the same strategy as in section~\ref{proof:max}. After applying the  Cauchy Schwartz inequality and making a change of variable $t\rightarrow Rt$ in the integral involving the function $g$ our job is reduced to proving the $L^{p_2}(\R^n)$ boundedness of the operator
$$\sup_{R>0}\int_0^{u}|H_{R,t}^{\beta}g(\cdot)|^2 dt,$$
where $H_{R,t}^{\beta}g(x)=\int_{\R^n}\psi_0^2\left(\frac{|\eta|^2}{R^2}\right)\left(1-t^2-\frac{|\eta|^2}{R^2}\right)_+^{\beta-1}\hat{g}(\eta)e^{2\pi ix.\eta} d\eta$. 

The term with the function $f$ is dealt with in a similar way as in the previous case. Note that $H_{R,t}^{\beta}g$ is similar to $S^{\beta,t}_{j,R}f$ with the term $\psi_0^2\left(\frac{|\eta|^2}{R^2}\right)$ on the Fourier transform side in place of $\psi\left(2^j\left(1-\frac{|\cdot|^2}{R^2}\right)\right)$ in $S^{\beta,t}_{j,R}f$. Since the function $\psi_0^2\left(\frac{|\eta|^2}{R^2}\right)$ is compactly supported and smooth with its support lying away from the origin, we can use arguments similar to the ones used in proving the estimate \eqref{easypart} and Proposition \ref{last} to deduce that the maximal function  $$\sup_{R>0}\left|\int_0^{u}|H_{R,t}^{\beta}g(\cdot)|^2 dt\right|$$
is bounded on $L^{p_2}(\R^n)$ for $\beta>\alpha(p_2)+1/2$ and $p_2=2$ or $p_2\geq \mathfrak p_n$. 

Combining the estimates together we get the desired result for the maximal function $\mathcal{B}^{\alpha}_{0,*}(f,g).$ 
\qed

%%%%%%%%%%%%%%%%%%%%%%%%%%%%%
\section{Proof of Theorem \ref{dim1}: The case of dimension $n=1$}\label{proof:br}

In this section we address the $L^p$ boundedness of the maximal function $\mathcal{B}^{\alpha}_*(f,g)$ in dimension $n=1$.

Recall that in view of the decomposition of the Bochner-Riesz operator given in section~\ref{sec:dec}, it is enough to establish the desired $L^p$ estimates for the square functions in the following inequality.
\begin{eqnarray*}
\|T_{*,R}^{\alpha}(f,g)\|_{L^{p}(\R)}
	&\lesssim & 2^{-\frac{j}{4}}\left\|\sup_{R>0}\left(\int_0^{\sqrt{2^{-j+1}}}|{S}_{j,\beta}^{R,t}f(\cdot)t^{2\delta+1}|^2 dt\right)^{1/2}\right\|_{L^{p_1}(\R)}\\ &&\left\|\left(\sup_{R>0}R_j^{-1}\int_0^{R_j}|B_t^{\delta}g(x)|^2dt\right)^{1/2}\right\|_{L^{p_2}(\R)},
	\end{eqnarray*}
where $\alpha=\beta+\delta$.

Invoking the $L^p$ boundedness of the square function for Bochner Riesz operator in dimension $n=1$ (see \cite{KanSun} Theorem B, page 361) and using the techniques of Lemma \ref{steinlemma}, it is easy to verify that the estimate  \begin{equation}\label{steinsquare}\left\|\left(\sup_{R>0}R_j^{-1}\int_0^{R_j}|B_t^{\delta}g(x)|^2dt\right)^{1/2}\right\|_{L^{p_2}(\R)}\lesssim \|g\|_{L^{p_2}(\R)}\end{equation} holds for $\delta>-1/2$ when $2\leq p_2<\infty$ and for $\delta>\frac{1}{p_2}-1$ when $1<p_2<2$.

Next, we have the following result for the other square function. 
\begin{theorem}\label{dim1theorem}  The inequality \begin{equation}\label{dim1square}
		\left\|\sup_{R>0}\left(\int_0^{\sqrt{2^{-j+1}}}|{S}_{j,\beta}^{R,t}f(\cdot)t^{2\delta+1}|^2 dt\right)^{1/2}\right\|_{L^{p_1}(\R)}\lesssim 2^{-j\tilde{\alpha} +j/4}\|f\|_{L^{p_1}(\R)} \end{equation} holds where $\alpha=\beta+\delta, \tilde{\alpha}=\text{min}\{\alpha,\delta+1/2\}$, $\beta>1/2$ for $2\leq p_2<\infty$ and $\beta>\frac{1}{p_1}$ for $1<p_1<2$.  
\end{theorem}
Note that the choice of $\delta$ in \eqref{steinsquare} implies that $\tilde{\alpha}>0$ for all $1<p_2< \infty$. 

The following lemma provides us with an important kernel estimate  in dimension $n=1$. This estimate is used in the proof of Theorem~\ref{dim1theorem}. 

\begin{lemma}\label{dim1kernel} Let $\psi\in C^{\infty}_0([1/2,2])$ and 
	$K_j(x):=\int_{\R}\psi(2^j(1-\xi^2))e^{2\pi ix\xi}d\xi$. Then for all $j\geq 2$ the following estimate holds
	$$|K_j(x)|\lesssim C 2^{-j}(1+|2^{-j}x|)^{-2},$$
	where $C= \sup_{x\in\R,~0\leq k\leq 2}\left|\frac{d^k \psi}{dx^k}\right|$.
\end{lemma}
\begin{remark}
We believe that the estimate in the lemma above should be known in the existing literature. However, we could not find it. Therefore, we have decided to present its proof here for the sake of completeness.	
\end{remark}
\begin{proof}
	Let $j$ be large enough so that $2^{-j}<<1/8$.
	Let $\chi$ be an even smooth function supported in $[-1/4,1/4]$ and identically one in $[-1/8,1/8]$.
	We can write $$K_j(x)=\int_{\R}\chi(\xi+1)\psi(2^j(1-\xi^2))e^{2\pi ix\xi}d\xi +\int_{\R}\chi(1-\xi)\psi(2^j(1-\xi^2))e^{2\pi ix\xi}d\xi .$$
	Observe that it is enough to show that $\tilde{K}_j(x)=\int_{\R}\chi(1-\xi)\psi(2^j(1-\xi^2))e^{2\pi ix\xi}d\xi$ satisfies the desired estimate.
	
	We perform change of variable $1-\xi\rightarrow \xi$ to the expression above and then one more time change variable $2^{j}\xi\rightarrow\xi$ to get that  $$\tilde{K}_j(x)=e^{2\pi ix}2^{-j}\int_{\R}\chi(2^{-j}\xi)\psi(\xi(2-2^{-j}\xi))e^{-2\pi i2^{-j}x\xi}d\xi.$$
	%We will analyse the integral $$\int_{\R}\chi(2^{-j}\xi)\psi(\xi(2-2^{-j}\xi))e^{-2\pi iy\xi}d\xi.$$

	Observe that $\psi(\xi(2-2^{-j}\xi))\chi(2^{-j}\xi)$ is a smooth function supported in $[-2^{j-2},2^{j-2}]$. Now we see that the function  $\psi(\xi(2-2^{-j}\xi))\chi(2^{-j}\xi)$ is of compact support with its support lying inside a set independent of $j$. We use an integration by parts argument to prove our claim.

	We already know that support of $\psi$ is $[1/2,2]$.
	Let us look at the function $ u(\xi)=\xi(2-2^{-j}\xi)$.
	It defines a parabola with its zeros at $0$ and $2^{j+1}$. One can show  that $2^j$ is the point where it attains its maximum and between $0$ to $2^j$ it is a strictly increasing function. 
	
	Further note that $u(2)=4(1-2^{-j})>2$ and 
	$u(\frac{1}{4})=\frac{1}{2}(1-2^{-j-3})<1/2$ for all $j\geq 8$. 
	Therefore $\{\xi:u(\xi)\subset [1/2,2]\}\subset [1/4,2]$.
	Hence the support of $\psi(\xi(2-2^{-j}\xi))\chi(2^{-j}\xi)$ is a subset of $[1/4,2]$ for all $j\geq 8$. 
	
	For $y\neq 0$, consider
	\begin{eqnarray*}
		&&\int_{\R}\chi(2^{-j}\xi)\psi(\xi(2-2^{-j}\xi))e^{-2\pi iy\xi}d\xi\\
		&=&-\frac{1}{2\pi iy}\int_{\R}\chi(2^{-j}\xi)\psi(\xi(2-2^{-j}\xi))\left(\frac{d}{d\xi}e^{-2\pi iy\xi}\right)d\xi \\ 
		&=&\frac{1}{2\pi iy}\left(\int_{\R}2^{-j}\chi'(2^{-j}\xi)\psi(\xi(2-2^{-j}\xi))e^{-2\pi iy\xi}d\xi+\int_{\R}\chi(2^{-j}\xi)(2-2^{-j+1}\xi)\psi'(\xi(2-2^{-j}\xi))e^{-2\pi iy\xi} d\xi\right).	
	\end{eqnarray*}
	Notice that the term $$\left|\int_{\R}2^{-j}\chi'(2^{-j}\xi)\psi(\xi(2-2^{-j}\xi))e^{-2\pi iy\xi}d\xi+\int_{\R}\chi(2^{-j}\xi)(2-2^{-j+1}\xi)\psi'(\xi(2-2^{-j}\xi))e^{-2\pi iy\xi} d\xi\right|$$
	is bounded by a uniform constant $C$ with respect to $j$. 
	 
	Similarly, the integration by parts argument applied to each of the integrals above gives us that 
	 $$\left|\int_{\R}\chi(2^{-j}\xi)\psi(\xi(2-2^{-j}\xi))e^{-2\pi iy\xi}d\xi\right|\lesssim C (1+y^2)^{-1}.$$
	This completes the proof.
\end{proof}

\noindent
{\bf Proof of Theorem \ref{dim1theorem}:} We shall skip detail in the proof and point out only the key steps as the proof uses the same strategy as in the  previous sections. 

Recall that $${S}_{j,\beta}^{R,t}f(x)=\int_{\R}\psi\left(2^j\left(1-\frac{\xi^2}{R^2}\right)\right)\left(1-t^2-\frac{\xi^2}{R^2}\right)_+^{\beta-1}\hat{f}(\xi)e^{2\pi ix\xi} d\xi.$$
		It is easy to verify that the kernel estimate from Lemma \ref{dim1kernel} implies that  \begin{equation}\label{100}\|\sup_{R>0}|B^{\psi}_{2^{-j},R}f|\|_{L^{p_1}(\R)}\leq C \|f\|_{L^{p_1}(\R)}
		\end{equation}
		holds for all $j\geq 2$ and $1<p_1\leq \infty$. 
		
		Fix $0<\epsilon_0<<1$ and consider the case
		when $0\leq t^2<2^{-j-1-\epsilon_0}$. In this case we proceed the same way as in Case 1 of Theorem \ref{maxsquare}. Using the estimate \eqref{100} we get that  $$\left\|\sup_{R>0}\left(\int_0^{\sqrt{2^{-j-1-\epsilon_0}}}|{S}_{j,\beta}^{R,t}f(\cdot)t^{2\delta+1}|^2 dt\right)^{1/2}\right\|_{L^{p_1}(\R)}\lesssim 2^{-j\alpha+j/4}\|f\|_{L^{p_1}(\R)}$$
		holds for all $\beta>1/2$.
		
		Next, for $2^{-j-1-\epsilon_0}<t^2\leq 2^{-j+1}$, we use the estimate~\eqref{100} again and follow the idea of Proposition \ref{last} to deduce the following estimate  $$\left\|\sup_{R>0}\left(\int^{\sqrt{2^{-j+1}}}_{\sqrt{2^{-j-1-\epsilon_0}}}\left|{S}_{j,\beta}^{R,t}f(\cdot)t^{2\delta+1}\right|^2 dt\right)^{1/2}\right\|_{L^{p_1}(\R)}\lesssim 2^{-j(\delta+1/2)+j/4}\|f\|_{L^{p_1}(\R)}$$
		for all $\beta>1/2$ if $2\leq p_1$ and $\beta>1/p_1$ if $1<p_1<2$. 
	
	\qed	
\section*{Acknowledgement} The second author acknowledges the financial support by Science and Engineering Research Board (SERB), Government of India, under the grant MATRICS: MTR/2017/000039/Math. 
%%%%%%%%%%%%%%%%%%%%%%%%%%%%%%%%%%%%

\end{document}